\documentclass[final]{amsart}

\usepackage{booktabs}
\usepackage{fullpage}
\usepackage{tristan}

\usepackage{algorithm}
\usepackage{algpseudocode}

\usepackage{pgfplots}
\usepackage{pgfplotstable}
\usepackage{subcaption}
\usepackage{graphicx}
\usepackage{placeins}

\definecolor{softorange}{RGB}{238,145,75}
\definecolor{softorange2}{RGB}{251,188,126}
\definecolor{softblue}{RGB}{126,169,199}
\definecolor{softteal}{RGB}{129,190,178}
\definecolor{softpurple}{RGB}{178,157,204}
\definecolor{softgrey}{RGB}{190,190,190}
\definecolor{darkorange}{RGB}{180,86,36}
\definecolor{darkgrey}{RGB}{70,70,70}

\newcounter{numexample}

\newcommand{\CurrentExample}{Example~\thenumexample}

\newcommand{\ExampleSubsection}[2]{%
  \refstepcounter{numexample}%
  \label{#2-exno}%
  \subsection{\CurrentExample: #1}%
  \label{#2}%
}

\usepackage{tikz}
\usetikzlibrary{arrows.meta,positioning,fit,calc,backgrounds,patterns}
\usepackage[
  giveninits=true,
  eprint=false,
  url=false,
  style=alphabetic,
  maxbibnames=99
]{biblatex}
\AtBeginBibliography{\footnotesize}

\author[1,2]{Tristan Pryer}

\address{$^1$ Institute for Mathematical Innovation\\ University of
  Bath, Bath, UK. $^2$ Department of Mathematical Sciences
  \\ University of Bath, Bath, UK.}

\begin{document}

\title{Batched Wavefront Sweeps for Linear Transport Uncertainty Quantification}

\begin{abstract}
  We develop a graph-compatible batched wavefront formulation for
  upwind discontinuous Galerkin discretisations of linear transport.
  For sweepable discretisations, the cell equations form a block
  lower-triangular system under a topological ordering of the directed
  upwind dependency graph. This causal structure can be shared across
  families of fixed-source problems whose local operators, sources and
  inflow data vary. We exploit this observation by grouping channels
  with a common dependency graph into \emph{graph-compatible sweep
  classes} and carrying samples, right-hand sides, energy groups and
  sign-compatible angular ordinates as tensor dimensions within a
  common wavefront schedule.

  We prove that the resulting batched wavefront algorithm is
  algebraically equivalent to independent classical DG sweeps for
  every channel in a class, while exposing parallelism simultaneously
  across wavefront cells and channel dimensions. We realise the
  formulation as a GPU tensor program based on batched cell-local
  solves and quantify its work, depth and storage requirements,
  including the throughput--memory tradeoff introduced by sample
  microbatching.  We further characterise the discrete
  parameter-to-observable map on fixed face-sign branches and show
  that reverse-mode differentiation through the wavefront algorithm
  reproduces the corresponding discrete adjoint action.

  Numerical experiments verify the expected DG convergence and adjoint
  consistency, demonstrate substantial GPU execution gains from sample
  batching, and illustrate the method in a material-shadowing
  uncertainty-quantification problem and a coupled multigroup
  C5G7/KAIST-inspired power-iteration workload.
\end{abstract}

\keywords{linear transport; discontinuous Galerkin methods; transport
  sweeps; wavefront algorithms; batched linear algebra; uncertainty
  quantification; automatic differentiation}

\maketitle

\section{Introduction}
\label{sec:introduction}

Discrete-ordinates transport calculations repeatedly solve
fixed-source advection--reaction problems. The solves occur over
angular ordinates and energy groups and inside source, power and
multiphysics iterations; further axes arise from multiple right-hand
sides and parameter samples. Classical deterministic methods exploit
the causal structure of the streaming operator through an upwind
transport sweep \cite{LewisMiller1984,AdamsLarsen2002}. Sampling,
stochastic collocation, quasi-Monte Carlo and multilevel Monte Carlo
methods add an ensemble of such calculations
\cite{XiuHesthaven2005,NobileTemponeWebster2008,Owen2013MonteCarlo,
  Giles2008}. Repeated transport solves are therefore a central cost
in neutron-transport uncertainty quantification
\cite{FichtlPrinja2011StochasticCollocation,
  GrahamParkinsonScheichl2018,CoxEtAl2022NTE, PasmannEtAl2022iQMC};
stochastic formulations provide a complementary probabilistic
description of the transport process
\cite{HortonKyprianouVillemonais2020NTEI,
  HarrisHortonKyprianou2020NTEII}.

For an upwind discontinuous Galerkin (DG) discretisation, each cell
depends only on inflow data and on traces already computed in upwind
neighbours. When the directed cell-dependency graph is acyclic, a
topological ordering makes the global matrix block lower triangular and
the transport sweep is block forward substitution. This structure goes
back to the original DG transport schemes
\cite{ReedHill1973,LesaintRaviart1974}; the stability and convergence
properties used below are classical
\cite{JohnsonPitkaranta1986,Cockburn1998IntroDG,CockburnShu2001,
CockburnDongGuzman2008}.

The transport-sweep literature addresses spatial scheduling, dependency
graphs, prioritisation, partitioning and cycle treatment
\cite{KochBakerAlcouffe1992KBA,BakerKoch1998,Pautz2002,
PlimptonEtAl2005,PautzBailey2017,HautEtAl2019,
VermaakEtAl2021}. Results on sweepability and cycle-free meshes
clarify when a pure topological sweep exists
\cite{CamminadyFrank2018Sweepable,AdamsEtAl2019OptimalSweeps,
calloo2026cycle,PeterEtAl2022SweepArepo}. Existing many-core and GPU
implementations also expose parallelism over spatial cells, angular
directions and energy groups
\cite{GongEtAl2011GPUTransport,GongEtAl2012UnstructuredGPU,
DeakinEtAl2018Parallelism,SearlesEtAl2019MiniSweep,
ZhangEtAl2025STRAUM,MorganEtAl2026OneCellGPU}. In particular,
angular and group aggregation within a common transport schedule is
already used in large-scale deterministic sweep algorithms
\cite{AdamsEtAl2019OptimalSweeps,VermaakEtAl2021,
AndrsEtAl2026OpenSn}.

Embedded ensemble propagation provides a complementary mechanism for
advancing parameter samples simultaneously, sharing geometric and
algebraic data while accommodating sample-dependent operators
\cite{PhippsEtAl2017EmbeddedEnsemble, DEliaEtAl2018EnsembleGrouping,
  LiegeoisEtAl2020EnsembleGMRES}. Batched dense and sparse linear
algebra similarly targets collections of independent small systems on
accelerator architectures
\cite{HaidarEtAl2015BatchedGPU,AbdelfattahEtAl2016BatchedGEMM,
  LiegeoisEtAl2023BatchedSparse}. Within deterministic transport
recent work has developed space--angle--energy sweep kernels,
vectorised-ordinate kernels sharing an octant schedule, and
performance-portable GPU sweeps based on batched local linear algebra
\cite{SuauEtAl2024SharedMemory,SuauEtAl2025Vectorized,
  SuauEtAl2025Batched}.

The closest antecedents address complementary subsets of the present
problem. In the embedded-ensemble formulations cited above, parameter
samples are propagated through a shared algebraic layout and the local
operator may vary between samples, but the transport-specific
dependency graph is not used to construct the ensemble partition.
Deterministic aggregated sweeps instead reuse a spatial schedule over
angle and energy-group channels, but stochastic sample batching and
branchwise differentiation are not principal components of those
formulations. The contribution here is a graph-compatible tensor
formulation that combines these axes while permitting the local DG
blocks, source terms and inflow data to vary across every channel.

Specifically, we call fixed-source channels \emph{graph compatible}
when their interior-face upwind signs induce the same directed cell
graph. Any one topological ordering of that graph may then be reused
for every channel in the corresponding \emph{sweep class}. Within a
class, parameter samples, right-hand sides, energy groups and
sign-compatible ordinates are carried in one tensor layout. Coupled
group and angular contributions are incorporated when forming the
fixed source, so the active channels remain independent during each
transport inversion. The algorithms and experiments below use
graph-compatible classes.  Section~\ref{sec:analysis} also introduces
the broader notion of schedule compatibility, in which the union of
the channel dependency graphs is acyclic and therefore admits a common
topological ordering.  The formulation identifies fixed-face-sign
regions on which differentiation of the wavefront program agrees with
the assembled discrete adjoint. Section~\ref{sec:numerics} quantifies
the resulting execution-layout speedup and memory--throughput
tradeoff.

Figure~\ref{fig:intro_tensorised_sweep} illustrates the construction.
A classical sweep stores one DG coefficient vector per cell. In the
tensorised sweep, the spatial wavefront ordering is unchanged, but
each cell state carries leading indices for the channels in its sweep
class. The cell updates on a wavefront are independent over both cells
and batch entries. The graph-based batching principle also extends to
sweepable polygonal and polyhedral meshes, where upwind traces and a
topological ordering are obtained from the corresponding
face-connectivity graph
\cite{calloo2026cycle,calloo2026diffusion,calloo2026transport}. We
concentrate on Cartesian meshes because their regular connectivity
permits upwind traces to be implemented through tensor shifts,
reducing indexing overhead and providing an efficient
accelerator-oriented realisation. Channels with different face-sign
patterns use separate wavefront schedules.

\begin{figure}[h!]
\centering
\resizebox{\textwidth}{!}{
\begin{tikzpicture}[
  x=1cm,
  y=1cm,
  font=\small,
  text=black,
  >={Latex[length=2.5mm,width=1.8mm]},
  arrow/.style={->, thick, draw=black!75},
  karrow/.style={->, very thick, draw=black!60, opacity=0.5},
  panel/.style={
    draw=black!25,
    rounded corners=4pt,
    thick,
    fill=black!2
  },
  cell/.style={
    draw=black!40,
    fill=white,
    minimum width=0.34cm,
    minimum height=0.34cm,
    inner sep=0pt
  },
  activecell/.style={
    draw=red!80!black,
    very thick,
    fill=red!8,
    minimum width=0.34cm,
    minimum height=0.34cm,
    inner sep=0pt
  },
  wavecell/.style={
    draw=black!40,
    fill=blue!25,
    minimum width=0.34cm,
    minimum height=0.34cm,
    inner sep=0pt
  },
  oldwavecell/.style={
    draw=black!35,
    fill=blue!9,
    minimum width=0.34cm,
    minimum height=0.34cm,
    inner sep=0pt
  },
  box/.style={
    draw=black!45,
    rounded corners=2pt,
    thick,
    fill=white,
    align=center,
    inner sep=4pt
  },
  algbox/.style={
    draw=black!55,
    rounded corners=2pt,
    thick,
    fill=yellow!18,
    align=center,
    inner sep=4pt
  },
  tensorface/.style={
    draw=black!45,
    thick
  }
]

\draw[panel] (0.0,0.0) rectangle (5.95,6.15);
\draw[panel] (6.25,0.0) rectangle (12.95,6.15);
\draw[panel] (13.25,0.0) rectangle (20.05,6.15);

\node[anchor=west, font=\bfseries] at (0.25,5.78) {(a) Classical sweep};
\node[anchor=west, font=\bfseries] at (6.50,5.78) {(b) Tensorised sweep};
\node[anchor=west, font=\bfseries] at (13.50,5.78) {(c) Batched wavefront update};

\begin{scope}[shift={(0.65,2.38)}]
  \foreach \i in {0,...,5}{
    \foreach \j in {0,...,5}{
      \pgfmathtruncatemacro{\s}{\i+\j}
      \ifnum\s=2
        \node[oldwavecell] at (0.43*\i,0.43*\j) {};
      \else
        \ifnum\s=3
          \node[wavecell] at (0.43*\i,0.43*\j) {};
        \else
          \node[cell] at (0.43*\i,0.43*\j) {};
        \fi
      \fi
    }
  }

  \node[activecell] at (0.43*3,0.43*2) {\scriptsize $K$};

  \draw[arrow] (-0.12,-0.12) -- (2.48,2.48);
  \node[font=\scriptsize, anchor=west] at (2.56,2.30) {sweep direction};

  \draw[karrow] (1.46,0.86) -- (3.18,0.86);
\end{scope}

\node[box, minimum width=1.95cm] at (4.58,3.24)
{$u_K\in\mathbb{R}^{n_p}$};

\node[font=\scriptsize, align=center] at (4.58,2.53)
{$n_p$ DG coefficients};

\node[font=\footnotesize, align=center, text width=4.9cm] at (2.98,1.05)
{one fixed-source channel\\one coefficient vector per cell};

\begin{scope}[shift={(6.90,2.38)}]
  \foreach \i in {0,...,5}{
    \foreach \j in {0,...,5}{
      \pgfmathtruncatemacro{\s}{\i+\j}
      \ifnum\s=2
        \node[oldwavecell] at (0.43*\i,0.43*\j) {};
      \else
        \ifnum\s=3
          \node[wavecell] at (0.43*\i,0.43*\j) {};
        \else
          \node[cell] at (0.43*\i,0.43*\j) {};
        \fi
      \fi
    }
  }

  \node[activecell] at (0.43*3,0.43*2) {\scriptsize $K$};

  \draw[arrow] (-0.12,-0.12) -- (2.48,2.48);

  \draw[karrow] (1.46,0.86) -- (3.18,0.86);
\end{scope}

\node[box, minimum width=2.35cm] at (10.83,3.24)
{$\vec U_K\in\mathbb{R}^{n_p\times n_\alpha}$};

\draw[tensorface, fill=green!20]  (9.95,4.00) rectangle (10.59,4.90);
\draw[tensorface, fill=blue!17]   (10.20,4.18) rectangle (10.84,5.08);
\draw[tensorface, fill=yellow!30] (10.45,4.36) rectangle (11.09,5.26);
\draw[tensorface, fill=orange!22] (10.70,4.54) rectangle (11.34,5.44);

\node[font=\scriptsize, align=left, anchor=west] at (11.58,4.76)
{sample\\angle\\group\\RHS};

\node[font=\scriptsize, align=center] at (10.83,2.53)
{$n_\alpha$ batch channels};

\node[font=\footnotesize, align=center, text width=5.4cm] at (9.60,1.05)
{same spatial sweep\\batch axes carried by each cell state};

\begin{scope}[shift={(13.70,2.38)}]
  \foreach \i in {0,...,5}{
    \foreach \j in {0,...,5}{
      \pgfmathtruncatemacro{\s}{\i+\j}
      \ifnum\s=2
        \node[oldwavecell] at (0.43*\i,0.43*\j) {};
      \else
        \ifnum\s=3
          \node[wavecell] at (0.43*\i,0.43*\j) {};
        \else
          \node[cell] at (0.43*\i,0.43*\j) {};
        \fi
      \fi
    }
  }

  \node[activecell] at (0.43*3,0.43*2) {\scriptsize $K$};

  \draw[arrow] (-0.12,-0.12) -- (2.48,2.48);
\end{scope}

\draw[karrow] (15.16,3.24) -- (16.50,4.58);
\draw[karrow] (15.16,3.24) -- (16.50,3.82);
\draw[karrow] (15.16,3.24) -- (16.50,3.06);

\node[algbox, minimum width=3.05cm] at (18.05,5.18)
{$n_\alpha$ batched local solves};

\foreach \x/\y/\shade in {
  16.74/4.60/35,
  17.46/4.60/30,
  18.18/4.60/24,
  18.90/4.60/18,
  16.74/3.84/35,
  17.46/3.84/30,
  18.18/3.84/24,
  18.90/3.84/18,
  16.74/3.08/35,
  17.46/3.08/30,
  18.18/3.08/24,
  18.90/3.08/18
}{
  \draw[draw=black!55, fill=yellow!\shade, thick] (\x-0.27,\y-0.27) rectangle ++(0.54,0.54);
  \draw[black!45] (\x,\y-0.27) -- ++(0,0.54);
  \draw[black!45] (\x-0.27,\y) -- ++(0.54,0);
}

\foreach \y/\shade in {4.60/12,3.84/12,3.08/12}{
  \draw[draw=black!30, fill=yellow!\shade, thick] (19.35,\y-0.27) rectangle ++(0.54,0.54);
  \draw[black!28] (19.62,\y-0.27) -- ++(0,0.54);
  \draw[black!28] (19.35,\y) -- ++(0.54,0);
}
\node[font=\Large] at (19.17,3.84) {$\cdots$};

\node[box, minimum width=4.55cm] at (16.95,1.55)
{$A_{\alpha,K}U_{\alpha,K}=b_{\alpha,K},
\qquad \alpha=1,\dots,n_\alpha$};

\node[font=\footnotesize, align=center, text width=5.05cm] at (16.95,0.80)
{independent over wavefront cells\\and batch index $\alpha$};

\end{tikzpicture}%
}
\caption{Tensorisation of the upwind DG sweep. Left: a classical fixed-source sweep
updates one local DG coefficient vector $u_K\in\mathbb{R}^{n_p}$ in each cell $K$.
Middle: the same wavefront ordering is retained for all channels in one sweep class,
but each cell carries a tensor-valued state
$\vec U_K\in\mathbb{R}^{n_p\times n_\alpha}$, where $n_\alpha$ collects the batch
channels associated with samples, sign-compatible angular ordinates, energy groups
and right-hand sides. Right: on a fixed wavefront, the local systems
$A_{\alpha,K}U_{\alpha,K}=b_{\alpha,K}$ are independent over cells on
the wavefront and over batch index $\alpha$. Angular ordinates with a different upwind
sign pattern use a different wavefront schedule.}
\label{fig:intro_tensorised_sweep}
\end{figure}
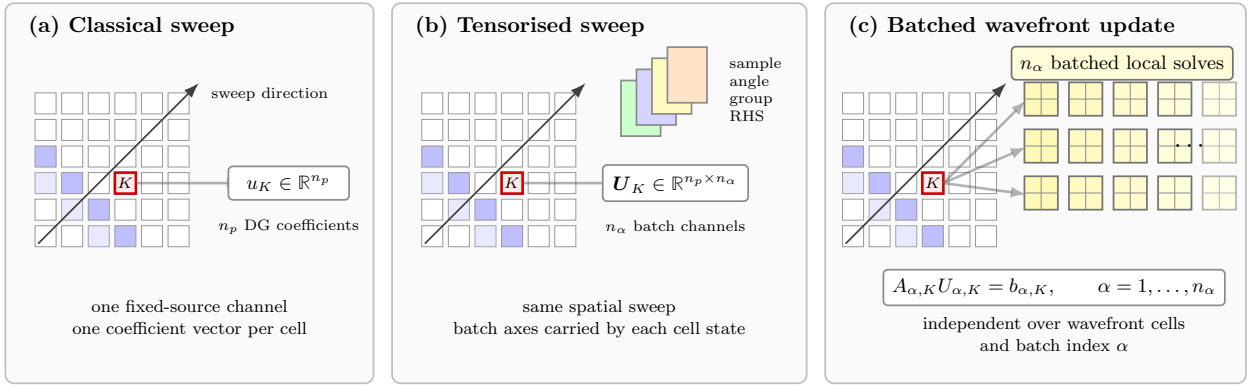

The paper makes the following contributions.

\begin{enumerate}
\item We distinguish graph compatibility from the more general notion
of schedule compatibility, give a face-sign signature for constructing
graph-compatible sweep classes, and lift cellwise block forward
substitution to a multi-axis tensor program. Local matrices, sources
and inflow data may vary across the channels sharing the spatial
schedule.

\item We give a Cartesian realisation in which wavefronts retain the
causal spatial order, upwind couplings are tensor shifts and the
cell-local systems are evaluated by batched dense kernels. Samples,
right-hand sides, sign-compatible ordinates and energy-group channels
within each fixed-source stage may be fused or microbatched according
to the available memory.

\item We establish algebraic equivalence with independent classical DG
sweeps and give the associated work, depth and storage model. We then
identify fixed-face-sign regions on which the discrete solver is smooth
and show that reverse-mode differentiation of the wavefront program
agrees with the assembled discrete adjoint action.

\item We check discretisation and implementation errors, measure
throughput and peak device memory as the active batch axes grow, verify
differentiated observables against an independently assembled triangular
system, and test simultaneous sample--angle--group batching in a
fixed-work multigroup power iteration.
\end{enumerate}

The graph-compatible construction is not specific to GPU execution.
It applies whenever a collection of block-triangular discrete problems
shares a dependency graph while retaining channel-dependent local
operators, couplings and right-hand sides. The GPU tensor
implementation studied here is one realisation of this structure;
other realisations include vectorised or batched CPU kernels,
performance-portable accelerator back ends, and on-node batching
within a distributed-memory transport sweep.

The contribution is complementary to spatial parallel-sweep algorithms.
In a distributed-memory implementation, a sweep class can be carried
inside a KBA or unstructured schedule: the established method supplies
the spatial partition and communication order, and the present tensor
layout supplies the on-node channel parallelism. It is also orthogonal
to diffusion synthetic acceleration and related methods, which reduce
the number of outer source iterations
\cite{Larsen1984DSA,AdamsMartin1992DSA}. Sweep-based production and
proxy codes such as Denovo, KRIPKE and OpenSn illustrate the broader
setting in which these kernels arise
\cite{EvansEtAl2010Denovo,KunenBaileyBrown2015Kripke,
AndrsEtAl2026OpenSn}.

The implementation uses tensor-array operations and reverse-mode
automatic differentiation. Modern array systems provide accelerator
back ends, program tracing and differentiated linear algebra
\cite{PaszkeEtAl2019PyTorch,FrostigJohnsonLeary2018Tracing,
BaydinEtAl2018ADSurvey}. Automated discrete adjoints and
differentiable finite-element solvers are well established
\cite{FarrellEtAl2013DolfinAdjoint,Giles2008MatrixDerivatives,
BlondelEtAl2022ImplicitDiff,XueEtAl2022JAXFEM}. Transport-specific
developments include differentiable GPU-based Boltzmann solvers and
differentiable Monte Carlo neutron transport
\cite{ShangEtAl2025JAXBTE,DengEtAl2026DifferentiableNeutron}.
Here the differentiated object is the deterministic block-triangular
DG sweep itself. On open fixed-face-sign regions, reverse-mode
differentiation through the wavefront program agrees exactly with the
assembled discrete adjoint; at a face-sign change, neighbouring
branches determine the limiting derivatives.

The sensitivity calculations below use mean-square pathwise derivatives
of discrete observables to construct derivative-based parameter rankings
\cite{SobolKucherenko2009DGSM,Cacuci2003}. The final multigroup
calculation evaluates combined sample--angle--group batching on a problem
constructed from C5G7 and KAIST data and geometry conventions
\cite{LewisSmithNa2001C5G7,SmithLewisNa2005C5G73D,Dahl2006PARTISN,
  Cho2000KAIST1A,KnightBryceHall2013}.

Beyond ensemble UQ, the same structure is therefore relevant to
multi-source calculations and to inverse, optimisation and
experimental design problems that require repeated differentiated
transport solves over a common sweep graph.

The rest of the paper is organised as follows.
Section~\ref{sec:model} defines the transport problem, discrete-ordinates
coupling and observables. Section~\ref{sec:analysis} gives the upwind DG
block form, dependency graph, sweep classes, tensor algorithm and cost
model. Section~\ref{sec:uq} analyses branchwise differentiability and
automatic differentiation of the discrete solver.
Section~\ref{sec:numerics} presents the numerical experiments, and
Section~\ref{sec:conclusions} draws the conclusions.

\section{Problem class, batch channels and observables}
\label{sec:model}

Let $(\Xi,\mathcal{A},\mathbb{P})$ be a probability space and let
$\xi:\Xi\to\mathbb{R}^{N_\xi}$ be a finite-dimensional random vector.
Let $d\in\qc{2,3}$ and let $D\subset\mathbb{R}^d$ be a bounded
polytopic Lipschitz domain with outward unit normal $\vec n$.
Coefficients, sources and boundary data may depend on $\xi$.

This section fixes the problem class needed by the batched wavefront
algorithm. The basic computational unit is a fixed-source transport
solve. Multigroup, angular and sampling structure enters by forming
many such fixed-source solves and grouping those with the same upwind
dependency graph into sweep classes.

\subsection{Fixed-source sweep kernel}

A single directional fixed-source solve has the standard
advection--reaction form used in discrete-ordinates transport
\cite{LewisMiller1984,AdamsLarsen2002},
\begin{equation}
  \vec b\cdot\nabla u(x)
  +
  c(x)u(x)
  =
  f(x),
  \qquad
  x\in D,
  \label{eq:model_sweep_kernel}
\end{equation}
with inflow condition
\begin{equation}
  u(x)=g(x),
  \qquad
  x\in\Gamma_-(\vec b),
  \label{eq:model_sweep_kernel_inflow}
\end{equation}
where the boundary is decomposed into
\begin{align}
  \Gamma_-(\vec b)
  &=
  \qc{x\in\partial D:\vec b\cdot\vec n(x)<0},
  \\
  \Gamma_+(\vec b)
  &=
  \qc{x\in\partial D:\vec b\cdot\vec n(x)>0},
  \\
  \Gamma_0(\vec b)
  &=
  \qc{x\in\partial D:\vec b\cdot\vec n(x)=0}.
  \label{eq:model_kernel_inflow_outflow_grazing}
\end{align}

For the scalar random examples, the fixed-source problem is
\begin{equation}
  \vec b(\xi)\cdot\nabla u(x,\xi)
  +
  c(x,\xi)u(x,\xi)
  =
  f(x,\xi),
  \qquad
  x\in D,
  \label{eq:model_scalar_transport}
\end{equation}
with inflow condition
\begin{equation}
  u(x,\xi)=g(x,\xi),
  \qquad
  x\in\Gamma_-(\vec b(\xi)).
  \label{eq:model_scalar_inflow}
\end{equation}
For a fixed realisation and a measurable subset
$\Gamma\subseteq\partial D$, define the weighted boundary space
\begin{equation}
L^2_w(\Gamma)
:=
\qc{
v:
\int_{\Gamma}\norm{w}\,\norm{v(x)}^2\,\mathrm ds<\infty
}.
\label{eq:model_weighted_inflow_space}
\end{equation}
We regard the boundary datum $g(\cdot,\xi)$ as a function on the fixed
boundary $\partial D$ and use only its restriction to the active inflow
set $\Gamma_-(\vec b(\xi))$. We assume for simplicity that
$\vec b(\xi)$ is spatially constant for each realisation and that
\begin{equation}
c(\cdot,\xi)\in L^\infty(D),
\qquad
f(\cdot,\xi)\in L^2(D),
\qquad
g(\cdot,\xi)\in L^2(\partial D).
\label{eq:model_scalar_data}
\end{equation}
In particular,
$g|_{\Gamma_-(\vec b(\xi))}
\in
L^2_{\norm{\vec b(\xi)\cdot\vec n}}
(\Gamma_-(\vec b(\xi)))$.
The use of the fixed space $L^2(\partial D)$ is convenient when the
inflow boundary changes with $\xi$. Whenever continuity,
Lipschitz continuity or differentiability with respect to $\xi$ is
invoked below, the data maps are assumed to have the corresponding
regularity into these fixed function spaces.

The reaction coefficient is assumed positive:
\begin{equation}
  c(x,\xi)\geq c_0(\xi)>0
  \qquad
  \text{for almost every }x\in D.
  \label{eq:model_scalar_positivity}
\end{equation}
When uniform estimates over the ensemble are required, we assume
\begin{equation}
  c_0(\xi)\geq c_\star>0
  \qquad
  \text{almost surely}.
  \label{eq:model_uniform_positivity}
\end{equation}

The scalar problem above represents one fixed-source channel. The
algorithmic setting of this paper is obtained by evaluating many such
channels. Independent fixed-source instances arise from Monte Carlo
samples, angular ordinates, energy groups and right-hand sides. Once
the right-hand side for a channel has been formed, the transport solve
has the form \eqref{eq:model_sweep_kernel}. Channels with the same
upwind cell-dependency graph can therefore be advanced together in one
batched wavefront sweep.

On a Cartesian mesh with spatially constant directions, equality of the
coordinate sign pattern of $\vec b$ gives the same upwind dependency
pattern. In two dimensions this gives quadrant classes; in three
dimensions it gives octant classes. The next subsection explains how
multigroup discrete-ordinates equations generate these fixed-source
channels.

\subsection{Multigroup discrete ordinates and source formation}

The multigroup discrete-ordinates equations used are obtained from the
steady linear Boltzmann transport equation by discretising angle with
a quadrature rule and energy with a finite number of groups
\cite{LewisMiller1984,AdamsLarsen2002,AndrsEtAl2026OpenSn}.  In this
setting the angular flux depends on position, direction and
energy. Scattering and fission couple directions and groups through
source terms, but after these sources have been formed each
group-ordinate equation has the same directional fixed-source form as
\eqref{eq:model_sweep_kernel}.

To that end, let
\begin{equation}
  \qc{(\vec\omega_m,w_m)}_{m=1}^{M},
  \qquad
  \vec\omega_m\in\mathbb{S}^{d-1},
  \qquad
  w_m>0,
  \label{eq:model_angular_quadrature}
\end{equation}
be the angular quadrature rule used in the discrete-ordinates
approximation. Energy groups are indexed by $\gamma=1,\dots,G$, and
$u_{\gamma,m}(x,\xi)$ denotes the angular flux in group $\gamma$ and
ordinate $m$. The scalar flux is
\begin{equation}
  \phi_\gamma(x,\xi)
  =
  \sum_{m=1}^{M}w_m u_{\gamma,m}(x,\xi).
  \label{eq:model_scalar_flux}
\end{equation}

Throughout the paper, the quadrature weights represent the normalised
angular measure and satisfy
\begin{equation}
\sum_{m=1}^{M}w_m=1.
\label{eq:model_angular_weight_normalisation}
\end{equation}
With this convention, isotropic scattering and fission sources are
independent of the ordinate index and are written as below. If a
solid-angle quadrature normalised by
$\sum_m w_m=\norm{\mathbb S^{d-1}}$ is used instead, the corresponding
reciprocal angular-measure factor is absorbed into the source
operators.

After scattering, fission or external sources have been formed, each
group-ordinate equation is a fixed-source solve:
\begin{equation}
  \vec\omega_m\cdot\nabla u_{\gamma,m}(x,\xi)
  +
  \Sigma_{t,\gamma}(x,\xi)u_{\gamma,m}(x,\xi)
  =
  Q_{\gamma,m}(x,\xi),
  \qquad
  x\in D,
  \label{eq:model_fixed_source_multigroup}
\end{equation}
with inflow condition
\begin{equation}
  u_{\gamma,m}(x,\xi)
  =
  g_{\gamma,m}(x,\xi),
  \qquad
  x\in\Gamma_{-,m},
  \label{eq:model_multigroup_inflow}
\end{equation}
where $\Gamma_{-,m} := \Gamma_-(\vec \omega_m)$. Thus
\eqref{eq:model_fixed_source_multigroup} is an instance of
\eqref{eq:model_sweep_kernel} with
\begin{equation}
  \vec b=\vec\omega_m,
  \qquad
  c=\Sigma_{t,\gamma},
  \qquad
  f=Q_{\gamma,m},
  \qquad
  g=g_{\gamma,m}.
  \label{eq:model_kernel_identification}
\end{equation}

In an isotropic multigroup setting, the scattering source is written
\begin{equation}
  (\mathcal{S}u)_\gamma(x,\xi)
  =
  \sum_{\gamma'=1}^{G}
  \Sigma_{s,\gamma'\to\gamma}(x,\xi)
  \phi_{\gamma'}(x,\xi),
  \label{eq:model_isotropic_scattering}
\end{equation}
and the fission-production source is
\begin{equation}
  (\mathcal{F}u)_\gamma(x,\xi)
  =
  \chi_\gamma(x,\xi)
  \sum_{\gamma'=1}^{G}
  \nu\Sigma_{f,\gamma'}(x,\xi)
  \phi_{\gamma'}(x,\xi).
  \label{eq:model_fission_source}
\end{equation}
The corresponding multigroup $k$-eigenvalue problem is
\begin{equation}
  \vec\omega_m\cdot\nabla u_{\gamma,m}
  +
  \Sigma_{t,\gamma}u_{\gamma,m}
  =
  (\mathcal{S}u)_\gamma
  +
  \frac{1}{k}(\mathcal{F}u)_\gamma.
  \label{eq:model_keff_problem}
\end{equation}
A power iteration forms scattering and fission sources from the current
scalar fluxes, solves the resulting fixed-source equations
\eqref{eq:model_fixed_source_multigroup}, and updates the fission
normalisation and $k$. Thus source formation couples groups and
ordinates between sweeps, but after $Q_{\gamma,m}$ has been formed,
each group-ordinate equation is again an independent fixed-source solve
within its sweep class.

\subsection{Random inputs and admissibility}

The random inputs are finite-dimensional. Uncertainty may enter through
the streaming direction, material coefficients, source terms or inflow
data. In the scalar setting, a random streaming direction is written as
\begin{equation}
  \vec b(\xi)=\beta(\xi)\vec\omega(\xi),
  \qquad
  \beta(\xi)>0,
  \qquad
  \vec\omega(\xi)\in\mathbb{S}^{d-1}.
  \label{eq:model_random_direction}
\end{equation}
In two-dimensional beam calculations we use the parametrisation
\begin{equation}
  \vec\omega(\xi)
  =
  \begin{pmatrix}
  \cos{\theta(\xi)}\\
  \sin{\theta(\xi)}
  \end{pmatrix}.
  \label{eq:model_random_angle}
\end{equation}
The coefficient, source and inflow data are allowed to depend on the
same finite-dimensional parameter vector $\xi$, provided the
fixed-source problem remains admissible. In particular, the scalar
reaction coefficient must satisfy the positivity condition
\eqref{eq:model_scalar_positivity}; when uniform estimates over the
ensemble are required, we assume \eqref{eq:model_uniform_positivity}.

In multigroup calculations, random inputs may enter the material-wise
removal, scattering and fission-production coefficients. Such
parametric material uncertainty is standard in deterministic
neutron-transport UQ; see, for example,
\cite{FichtlPrinja2011StochasticCollocation,
  GrahamParkinsonScheichl2018, PasmannEtAl2022iQMC}.We use
$\Sigma_{a,\gamma}$ for non-scattering removal and
$\nu\Sigma_{f,\gamma}$ for the fission-production coefficient
appearing in \eqref{eq:model_fission_source}. Removal is represented
through $\Sigma_{t,\gamma}$ and $\Sigma_{a,\gamma}$, whereas
$\nu\Sigma_{f,\gamma}$ enters the production source. Thus a
perturbation of $\nu\Sigma_f$ changes fission production at fixed
removal unless a corresponding removal perturbation is specified
separately; for example, it may represent a change in neutron yield at
fixed removal.

The admissibility conditions are
\begin{equation}
  \Sigma_{t,\gamma}(x,\xi)>0,
  \qquad
  \Sigma_{a,\gamma}(x,\xi)\geq 0,
  \qquad
  \Sigma_{s,\gamma'\to\gamma}(x,\xi)\geq 0,
  \qquad
  \nu\Sigma_{f,\gamma}(x,\xi)\geq 0.
  \label{eq:model_cross_section_admissibility}
\end{equation}
Here
$\Sigma_{s,\gamma\to\gamma'}$
denotes scattering out of group $\gamma$ into group $\gamma'$.
The loss coefficients are required to satisfy
\begin{equation}
\Sigma_{t,\gamma}(x,\xi)
\geq
\Sigma_{a,\gamma}(x,\xi)
+
\sum_{\gamma'=1}^{G}
\Sigma_{s,\gamma\to\gamma'}(x,\xi).
\label{eq:model_total_cross_section_consistency}
\end{equation}
Any positive difference represents additional removal not resolved by
the displayed absorption and scattering coefficients.

\subsection{Observables}

Let $u(\cdot,\xi)$ denote the solution of one or more fixed-source
transport problems associated with the parameter value $\xi$. We write a
scalar observable in the form
\begin{equation}
  Y(\xi)
  =
  J(u(\cdot,\xi),\xi).
  \label{eq:model_observable}
\end{equation}
Here $J$ denotes a functional of the transport solution and, possibly,
of the sampled coefficients, source data, inflow data or detector
geometry. Typical examples include outflow currents, volume averages,
absorption functionals, effective multiplication factors and pin-wise
fission production. Concrete choices are specified in the
numerical experiments.

\section{Tensorised upwind DG sweep}
\label{sec:analysis}

This section describes the discrete sweep structure used by the tensor
implementation. We first write the upwind DG equations for one
fixed-source channel and express the resulting system in cellwise block
form. The upwind numerical flux defines a directed cell dependency
graph. If this graph is acyclic, the assembled system is block lower
triangular after a topological ordering and can be solved by a
wavefront sweep. The tensorised algorithm applies this same block
forward substitution over all independent channels that share the same
dependency graph.

Throughout this section, one fixed-source channel is fixed unless batch
indices are explicitly displayed. Thus $\vec b$, $c$, $f$ and $g$
denote the direction, reaction coefficient, source and inflow data for
one sample, ordinate, group and right-hand side.

\subsection{Upwind DG equations and cell blocks}

Let $\mathcal{T}_h$ be a Cartesian mesh of $D$, with cells $K$. Let
$\mathcal{F}_h^{\mathrm{int}}$ and $\mathcal{F}_h^{\partial}$ denote
the sets of interior and boundary faces. For each cell $K$, let $\vec
n_K$ be the outward unit normal and define
\begin{align}
\partial K_+
&=
\qc{x\in\partial K:\vec b\cdot\vec n_K(x)>0},
\\
\partial K_-
&=
\qc{x\in\partial K:\vec b\cdot\vec n_K(x)<0},
\\
\partial K_0
&=
\qc{x\in\partial K:\vec b\cdot\vec n_K(x)=0}.
\end{align}
Faces in $\partial K_0$ are grazing faces. They carry zero normal flux
and create no upwind dependency.

For an integer $p\geq 0$, let
\begin{equation}
V_h
=
\qc{v_h\in L^2(D):v_h|_K\in \mathbb{Q}_p(K)\text{ for all }K\in\mathcal{T}_h},
\label{eq:analysis_vh}
\end{equation}
where $\mathbb{Q}_p(K)$ denotes the tensor-product polynomial space of degree
at most $p$ in each coordinate on $K$. On a non-grazing element
boundary face, the upwind trace of a trial function $w_h$ is
\begin{equation}
w_h^{\mathrm{up}}
=
\begin{cases}
w_h|_K, & \vec b\cdot\vec n_K>0,\\
w_h|_{K'}, & \vec b\cdot\vec n_K<0\text{ and }F=K\cap K',\\
g, & \vec b\cdot\vec n_K<0\text{ and }F\subset\partial D.
\end{cases}
\label{eq:analysis_upwind_trace}
\end{equation}
We use the classical upwind discontinuous Galerkin construction for
first-order transport \cite{ReedHill1973,LesaintRaviart1974,
  JohnsonPitkaranta1986,Cockburn1998IntroDG}. The upwind DG method is
to find $u_h\in V_h$ such that
\begin{equation}
a_h(u_h,v_h)=\ell_h(v_h)
\qquad
\text{for all }v_h\in V_h,
\label{eq:analysis_dg_problem}
\end{equation}
where
\begin{equation}
\begin{split}
a_h(w_h,v_h)
&=
\sum_{K\in\mathcal{T}_h}
\qb{
-\int_K
w_h\vec b\cdot\nabla v_h dx
+
\int_K
c w_h v_h dx
+
\int_{\partial K_+}
(\vec b\cdot\vec n_K)w_h v_h ds
}
\\
&\quad+
\sum_{K\in\mathcal{T}_h}
\int_{\partial K_-\setminus\partial D}
(\vec b\cdot\vec n_K)w_h^{\mathrm{up}}v_h ds,
\end{split}
\label{eq:analysis_dg_bilinear}
\end{equation}
and
\begin{equation}
\ell_h(v_h)
=
\sum_{K\in\mathcal{T}_h}
\int_K f v_h dx
-
\sum_{K\in\mathcal{T}_h}
\int_{\partial K_-\cap\partial D}
(\vec b\cdot\vec n_K)g v_h ds.
\label{eq:analysis_dg_linear}
\end{equation}
Physical inflow data are therefore imposed weakly through the boundary
flux.

The natural DG norm is
\begin{equation}
\Norm{v_h}_{\mathrm{DG}}^2
=
\int_D c v_h^2 dx
+
\frac{1}{2}
\sum_{F\in\mathcal{F}_h^{\mathrm{int}}}
\int_F \norm{\vec b\cdot\vec n_F}\jump{v_h}^2 ds
+
\frac{1}{2}
\int_{\partial D}
\norm{\vec b\cdot\vec n}v_h^2 ds,
\label{eq:analysis_dg_norm}
\end{equation}
where, for an interior face with fixed normal $\vec n_F$, $\jump{v_h}$
denotes the jump across the face. Grazing faces make no contribution.

\begin{proposition}[Discrete stability]
\label{prop:analysis_discrete_stability}
Assume that $\vec b$ is spatially constant and that $c(x)\geq c_0>0$
almost everywhere in $D$. Then, for all $v_h\in V_h$,
\begin{equation}
  a_h(v_h,v_h)=\Norm{v_h}_{\mathrm{DG}}^2.
\label{eq:analysis_dg_coercivity_identity}
\end{equation}
Consequently, the DG solution satisfies
\begin{equation}
\Norm{u_h}_{\mathrm{DG}}
\leq
C
\qp{
\Norm{f}_{L^2(D)}
+
\Norm{g}_{L^2_{\norm{\vec b\cdot\vec n}}(\Gamma_-)}
},
\label{eq:analysis_dg_stability}
\end{equation}
where $C$ depends on $c_0$ and $D$, but not on $h$.
\end{proposition}

\begin{proof}
The identity follows by taking $w_h=v_h$ in
\eqref{eq:analysis_dg_bilinear}, integrating the advection term by
parts elementwise, and using $\nabla\cdot\vec b=0$. The two traces on
each interior face combine with the upwind flux to give
$\frac{1}{2}\int_F\norm{\vec b\cdot\vec n_F}\jump{v_h}^2ds$, while boundary faces
give $\frac{1}{2}\int_F\norm{\vec b\cdot\vec n}v_h^2ds$. The stability
estimate follows by testing \eqref{eq:analysis_dg_problem} with $u_h$
and applying Cauchy--Schwarz and Young's inequality to the volume and
inflow terms.
\end{proof}

For a shape-regular, quasi-uniform Cartesian mesh family and an exact
solution satisfying
$u|_K\in H^{p+1}(K)$ on every cell, the classical upwind DG analysis
under the coefficient assumptions above gives the reference
transport-norm estimate
\begin{equation}
\Norm{ u-u_h}_{\mathrm{DG}}
\leq
C h^{p+1/2}
\Norm{ u}_{H^{p+1}(\mathcal{T}_h)},
\label{eq:analysis_dg_error}
\end{equation}
where $C$ is independent of $h$; see, for example,
\cite{JohnsonPitkaranta1986,CockburnDongGuzman2008}.

Let $n_p=(p+1)^d$ and let $\qc{\varphi_r^K}_{r=1}^{n_p}$ be a basis for
$\mathbb{Q}_p(K)$. With
\begin{equation}
u_h|_K(x)
=
\sum_{s=1}^{n_p}
U_{K,s}\varphi_s^K(x),
\label{eq:analysis_local_expansion}
\end{equation}
testing with the local basis functions gives one block equation per
cell:
\begin{equation}
A_K U_K
=
R_K
+
\sum_{K'\in\mathcal{U}(K)}
B_{K,K'}U_{K'}
+
G_K.
\label{eq:analysis_cell_block_equation}
\end{equation}
Here $\mathcal{U}(K)$ is the set of upwind neighbours of $K$. The
diagonal block, volume source, upstream coupling block and inflow
boundary vector are
\begin{align}
(A_K)_{rs}
&=
-\int_K
\varphi_s^K\vec b\cdot\nabla\varphi_r^K dx
+
\int_K
c\varphi_s^K\varphi_r^K dx
+
\int_{\partial K_+}
(\vec b\cdot\vec n_K)
\varphi_s^K\varphi_r^K ds,
\label{eq:analysis_local_diagonal_block}
\\
(R_K)_r
&=
\int_K f\varphi_r^K dx,
\label{eq:analysis_local_source_vector}
\\
(B_{K,K'})_{rs}
&=
-\int_F
(\vec b\cdot\vec n_K)
\varphi_s^{K'}\varphi_r^K ds
=
\int_F
\norm{\vec b\cdot\vec n_K}
\varphi_s^{K'}\varphi_r^K ds,
\label{eq:analysis_upwind_coupling_block}
\\
(G_K)_r
&=
-\int_{\partial K_-\cap\partial D}
(\vec b\cdot\vec n_K)g\varphi_r^K ds.
\label{eq:analysis_boundary_inflow_vector}
\end{align}
\footnote{The numbering of \eqref{eq:analysis_upwind_coupling_block}
is, of course, entirely coincidental.}  In
\eqref{eq:analysis_upwind_coupling_block}, $F=K\cap K'$ and $\vec
b\cdot\vec n_K<0$, so that $K'$ is upstream of $K$.

\begin{lemma}[Local block nonsingularity]
\label{lem:analysis_local_block_nonsingularity}
Assume that $\vec b$ is constant and $c(x)\geq c_0>0$ almost everywhere
on $K$. Then the local block $A_K$ is nonsingular.
\end{lemma}

\begin{proof}
For $V\in\mathbb{R}^{n_p}$, let
$v=\sum_s V_s\varphi_s^K$. Then
\begin{equation}
V^T A_K V
=
\int_K c v^2 dx
+
\frac{1}{2}
\int_{\partial K}
\norm{\vec b\cdot\vec n_K}v^2 ds,
\end{equation}
by the same elementwise integration-by-parts argument used above. If
$A_KV=0$, then the right-hand side vanishes, hence $v=0$ in $K$ by
positivity of $c$. Since the basis is linearly independent, $V=0$.
\end{proof}

The block equations \eqref{eq:analysis_cell_block_equation}, assembled
over all cells, are equivalent to the upwind DG formulation
\eqref{eq:analysis_dg_problem}: a test function supported on one cell
selects one block row, and summing the block rows reconstructs the
global variational problem.

\subsection{Dependency graph and sweep classes}
\label{subsec:analysis_sweep_classes}

The upwind fluxes define a directed dependency graph
\begin{equation}
  \mathcal{G}_{h,\vec b}
  =
  (\mathcal{V}_h,\mathcal{E}_{h,\vec b}),
\end{equation}
whose vertices are the cells of $\mathcal{T}_h$. For two distinct cells
$K$ and $K'$, there is a directed edge
\begin{equation}
  K'\to K
  \qquad
  \Longleftrightarrow
  \qquad
  \vec b\cdot\vec n_K<0
  \quad\text{on }F=K\cap K'.
  \label{eq:analysis_graph_edge}
\end{equation}
Thus an edge is directed from the upwind cell $K'$ to the downstream
cell $K$ whenever the local equation on $K$ uses the trace from $K'$.
Boundary inflow faces contribute prescribed data but do not create
cell-to-cell graph edges, while grazing faces create no
edge. Figure~\ref{fig:analysis_dependency_graph} illustrates this
construction for a Cartesian mesh.

\begin{figure}[h!]
\centering

\begin{tikzpicture}[
  x=1.15cm,
  y=1.15cm,
  >=Latex,
  line cap=round,
  line join=round,
  primal/.style={
    draw=darkgrey,
    line width=0.8pt
  },
  dualnode/.style={
    circle,
    fill=darkgrey,
    draw=darkgrey,
    inner sep=1.5pt
  },
  depedge/.style={
    ->,
    draw=darkorange,
    line width=1.0pt
  },
  flowarrow/.style={
    ->,
    draw=darkgrey,
    line width=1.2pt
  },
  inflow/.style={
    draw=softblue,
    line width=1.5pt
  },
  wavefront/.style={
    dashed,
    draw=softgrey,
    line width=0.8pt
  },
  lab/.style={
    font=\small,
    text=darkgrey
  },
  smalllab/.style={
    font=\scriptsize,
    text=darkgrey
  }
]

\begin{scope}

  \draw[primal] (0,0) rectangle (4,4);

  \foreach \i in {1,2,3}{
    \draw[primal] (\i,0) -- (\i,4);
  }

  \foreach \j in {1,2,3}{
    \draw[primal] (0,\j) -- (4,\j);
  }

  \fill[softorange2!38] (2,2) rectangle (3,3);
  \fill[softorange2!18] (1,2) rectangle (2,3);
  \fill[softorange2!18] (2,1) rectangle (3,2);

  \node[lab]      at (2.5,2.5) {$K$};
  \node[smalllab] at (1.5,2.5) {$K'$};
  \node[smalllab] at (2.5,1.5) {$K'$};

  \draw[inflow] (-0.04,0) -- (-0.04,4);
  \draw[inflow] (0,-0.04) -- (4,-0.04);
 \node[
    smalllab,
    text=softblue,
    rotate=90
  ] at (-0.32,2) {inflow};

  \node[
    smalllab,
    text=softblue
  ] at (2,-0.30) {inflow};

  \draw[flowarrow]
    (2.65,3.35) -- (3.55,3.80);

  \node[lab, anchor=west]
    at (3.58,3.80)
    {$\vec b$};

  \draw[
    ->,
    draw=darkgrey,
    line width=0.8pt
  ]
    (2.02,2.1) -- (1.72,2.1);

  \node[smalllab]
    at (1.78,2.25)
    {$\vec n_K$};

  \draw[
    ->,
    draw=darkorange,
    line width=1.1pt
  ]
    (1.60,2.5) -- (2.34,2.5);

  \draw[
    ->,
    draw=darkorange,
    line width=1.1pt
  ]
    (2.5,1.60) -- (2.5,2.34);

  \node[lab]
    at (2,-0.50)
    {Primal mesh $\mathcal{T}_h$};

\end{scope}

\begin{scope}[shift={(6.25,0)}]

  \draw[primal, opacity=0.16]
    (0,0) rectangle (4,4);

  \foreach \i in {1,2,3}{
    \draw[primal, opacity=0.16]
      (\i,0) -- (\i,4);
  }

  \foreach \j in {1,2,3}{
    \draw[primal, opacity=0.16]
      (0,\j) -- (4,\j);
  }

  \begin{scope}
    \clip (0,0) rectangle (4,4);

    \foreach \c in {
      1.5,2.5,3.5,4.5,5.5,
      6.5,7.5,8.5,9.5,10.5
    }{
      \draw[wavefront]
        (-1,{\c+2}) -- (6,{\c-12});
    }
  \end{scope}

  \fill[
    softorange2!38,
    draw=softorange,
    line width=0.8pt
  ]
    (2.08,2.08) rectangle (2.92,2.92);

  \foreach \i in {1,...,4}{
    \foreach \j in {1,...,4}{
      \node[dualnode]
        (v-\i-\j)
        at ({\i-0.5},{\j-0.5}) {};
    }
  }

  \foreach \i in {1,...,3}{
    \pgfmathtruncatemacro{\ip}{\i+1}
    \foreach \j in {1,...,4}{
      \draw[depedge]
        (v-\i-\j) -- (v-\ip-\j);
    }
  }

  \foreach \j in {1,...,3}{
    \pgfmathtruncatemacro{\jp}{\j+1}
    \foreach \i in {1,...,4}{
      \draw[depedge]
        (v-\i-\j) -- (v-\i-\jp);
    }
  }

  \node[
    lab,
    fill=softorange2!70,
    inner sep=1.5pt
  ] at (2.5,2.5) {$K$};

  \draw[flowarrow]
    (2.65,3.35) -- (3.55,3.80);

  \node[lab, anchor=west]
    at (3.58,3.80)
    {$\vec b$};

  \node[lab]
    at (2,-0.50)
    {Dependency graph $\mathcal{G}_{h,\vec b}$};

\end{scope}

\end{tikzpicture}

\caption{ An illustration of a Cartesian primal mesh and associated
  directed cell-dependency graph. Graph vertices are placed at cell
  barycentres and orange arrows point from upwind to downstream
  cells. The dashed grey lines indicate the admissible topological
  levels in the right panel.  }
\label{fig:analysis_dependency_graph}
\end{figure}
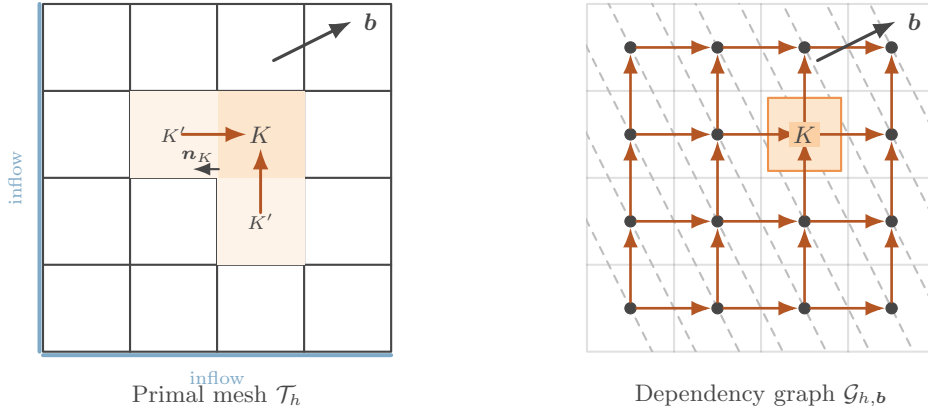

The dependency graph is determined by the upwind face signs. In
particular, every constant direction with $b_x>0$ and $b_y>0$ induces
the same directed graph as the example in
Figure~\ref{fig:analysis_dependency_graph}. This observation is the
basis for grouping several fixed-source problems into a common sweep
class.

\begin{definition}[Sweepability]
The pair $(\mathcal{T}_h,\vec b)$ is called \emph{sweepable} if
$\mathcal{G}_{h,\vec b}$ is acyclic. A \emph{sweep ordering} is any
topological ordering $K_1,\ldots,K_{N_{\mathrm{cell}}}$ of this graph,
so that every edge $K_i\to K_j$ satisfies $i<j$.
\end{definition}

On Cartesian meshes with spatially constant $\vec b$, the dependency
graph is acyclic. On more general meshes, or for spatially varying
streaming fields, cycles may occur. Dependency-graph construction,
topological sweeping and the treatment of cyclic dependencies are
standard components of deterministic transport sweep algorithms
\cite{Pautz2002,PlimptonEtAl2005,
  VermaakEtAl2021,CamminadyFrank2018Sweepable}. Such cycles can be
handled by collapsing strongly connected components into larger
blocks, or by using the sweep as part of an iterative solver or
preconditioner. The present analysis concerns the sweepable case.

For a fixed-source channel $\alpha$, write
\begin{equation}
  \mathcal{G}_{h,\alpha}
  :=
  \mathcal{G}_{h,\vec b_\alpha}
  =
  (\mathcal{V}_h,\mathcal{E}_{h,\alpha}).
\end{equation}

\begin{definition}[Graph and schedule compatibility]
Two channels $\alpha$ and $\beta$ are \emph{graph compatible} if
\begin{equation}
  \mathcal{E}_{h,\alpha}
  =
  \mathcal{E}_{h,\beta}.
\end{equation}
A collection $\mathcal{C}$ of channels is \emph{schedule compatible}
if the union graph
\begin{equation}
  \mathcal{G}^{\cup}_{h,\mathcal{C}}
  :=
  \qp{
    \mathcal{V}_h,
    \bigcup_{\alpha\in\mathcal{C}}
    \mathcal{E}_{h,\alpha}
  }
\end{equation}
is acyclic. In that case, every topological ordering of
$\mathcal{G}^{\cup}_{h,\mathcal{C}}$ is a valid ordering for every
channel in $\mathcal{C}$.
\end{definition}

Every graph-compatible collection is schedule compatible, although
distinct dependency graphs may also admit a common topological
ordering. The implementation below uses graph-compatible sweep
classes. Reuse of a common spatial sweep over angular or energy
channels is also exploited in deterministic transport implementations
\cite{AdamsEtAl2019OptimalSweeps, SuauEtAl2024SharedMemory,
  SuauEtAl2025Vectorized}.  The distinction here is that the class is
constructed from the common dependency graph and may additionally
carry sample-dependent local operators and data.

Let
\begin{equation}
  \alpha=(n,m,\gamma,r)
\end{equation}
denote the combined batch index for sample, angular ordinate, energy
group and right-hand side, respectively. Fix one orientation
$\vec n_F$ on every interior face and define the interior-face sign
signature
\begin{equation}
  \sigma^{\mathrm{int}}_\alpha
  :=
  \qp{
    \operatorname{sign}
    (\vec b_\alpha\cdot\vec n_F)
  }_{F\in\mathcal{F}^{\mathrm{int}}_h}.
\end{equation}
Under the standing assumption that each streaming direction is
spatially constant, equal signatures induce equal directed dependency
graphs. Boundary-face signs are not included in
$\sigma^{\mathrm{int}}_\alpha$: they determine the physical inflow
faces and hence the inflow vectors, but they do not create
cell-to-cell graph edges.

The graph-compatible partition is therefore obtained by grouping equal
sign signatures:
\begin{equation}
  \mathcal{I}_\alpha
  =
  \bigcup_{s\in\mathcal{S}}
  \mathcal{I}^{(s)}_\alpha,
  \qquad
  \mathcal{I}^{(s)}_\alpha
  \cap
  \mathcal{I}^{(s')}_\alpha
  =
  \emptyset
  \quad
  \text{for }s\neq s',
\end{equation}
where every class uses one chosen topological ordering of its common
graph. A direct implementation computes or hashes
$\sigma^{\mathrm{int}}_\alpha$ for each channel, groups equal
signatures, constructs one graph for each distinct signature, and
performs one topological sort per class. Assuming constant-time face
sign evaluations, the preprocessing cost is
\begin{equation}
  \mathcal{O}
  \qp{
    n_\alpha
    \norm{\mathcal{F}^{\mathrm{int}}_h}
    +
    \sum_{s\in\mathcal{S}}
    \qp{
      N_{\mathrm{cell}}
      +
      \norm{\mathcal{E}^{(s)}_h}
    }
  }.
\end{equation}

On a Cartesian mesh with spatially constant directions, the full
interior-face signature reduces to the coordinate sign pattern
\begin{equation}
  s(\alpha)
  =
  \qp{
    \operatorname{sign}(b_{\alpha,1}),
    \ldots,
    \operatorname{sign}(b_{\alpha,d})
  }
  \in
  \qc{-1,0,+1}^d
  \setminus\qc{(0,\ldots,0)}.
\end{equation}
The corresponding graph-compatible class is
\begin{equation}
  \mathcal{I}^{(s)}_\alpha
  =
  \qc{
    \alpha\in\mathcal{I}_\alpha:
    \operatorname{sign}(b_{\alpha,q})=s_q,
    \quad q=1,\ldots,d
  }.
\end{equation}
The value $s_q=0$ means that faces normal to coordinate direction $q$
are grazing for that channel and therefore create no dependency in
that coordinate direction.

For one sweep class, the solution coefficients may be stored as
\begin{equation}
U^{(s)}
\in
\mathbb{R}^{n_\alpha^{(s)}\times N_{\mathrm{cell}}\times n_p},
\qquad
n_\alpha^{(s)}=\norm{\mathcal{I}_\alpha^{(s)}}.
\label{eq:analysis_flat_solution_tensor}
\end{equation}
Throughout, $A_{\alpha,K}^{-1}b_{\alpha,K}$ denotes the action of
solving the local linear system. It does not require the explicit
inverse matrix to be formed. For each batch entry $\alpha$ and cell $K$, the local system is
\begin{equation}
A_{\alpha,K}U_{\alpha,K}
=
R_{\alpha,K}
+
\sum_{K'\in\mathcal U_\alpha(K)}
B_{\alpha,K,K'}U_{\alpha,K'}
+
G_{\alpha,K}.
\label{eq:analysis_batched_local_update}
\end{equation}
The matrices may depend on the sample, group, material and cell. The
dependency graph may depend on the angular direction; this is why the
update is applied separately within each sweep class.

On Cartesian meshes with a constant direction, the upstream neighbours
are coordinate shifts. If $\vec i=(i_1,\dots,i_d)$ and
$s_q=\operatorname{sign}(b_q)$, then the upwind neighbour in an active
coordinate direction $q$ with $s_q\neq0$ is $\vec i-s_q\vec e_q,$
provided that this cell lies in the mesh. Thus the neighbour terms in
\eqref{eq:analysis_batched_local_update} can be implemented by tensor
shifts over the active coordinate directions.

\subsection{Batched wavefront algorithm and exactness}

The loop over wavefronts is sequential because it follows the
transport ordering. Within a wavefront, the updates are independent
over cells in $\mathcal{W}^{(s)}_k$ and over batch indices
$\alpha$. The local right-hand sides and dense block solves can
therefore be evaluated as batched tensor operations as in Algorithm
\ref{alg:batched_wavefront_sweep}.

\begin{algorithm}[h!]
\caption{Batched tensor wavefront DG sweep by sweep class}
\label{alg:batched_wavefront_sweep}
\begin{algorithmic}[1]
\Require Sweep-class partition
$\mathcal{I}_\alpha=\bigcup_{s\in\mathcal{S}}\mathcal{I}_\alpha^{(s)}$,
wavefronts
$\mathcal{W}^{(s)}_1,\dots,\mathcal{W}^{(s)}_{N_W^{(s)}}$
for each class $s$, upwind neighbour sets $\mathcal{U}^{(s)}(K)$,
local blocks $A_{\alpha,K}$, coupling blocks $B_{\alpha,K,K'}$,
source vectors $R_{\alpha,K}$, and inflow vectors $G_{\alpha,K}$
\Ensure Batched DG coefficients $U_{\alpha,K}$

\ForAll{sweep classes $s\in\mathcal{S}$}
  \State Initialise the class tensor
  $U^{(s)}\in
  \mathbb{R}^{n_\alpha^{(s)}\times N_{\mathrm{cell}}\times n_p}$.
  \For{$k=1,\dots,N_W^{(s)}$}
    \ForAll{$K\in\mathcal{W}^{(s)}_k$ and
    $\alpha\in\mathcal{I}_\alpha^{(s)}$}
      \State Form
      \[
      b_{\alpha,K}
      =
      R_{\alpha,K}
      +
      \sum_{K'\in\mathcal U^{(s)}(K)}
      B_{\alpha,K,K'}U_{\alpha,K'}
      +
      G_{\alpha,K}.
      \]
      \State Solve the local system
      \[
      A_{\alpha,K}U_{\alpha,K}=b_{\alpha,K}.
      \]
    \EndFor
  \EndFor
\EndFor
\State Return $\qc{U^{(s)}}_{s\in\mathcal{S}}$.
\end{algorithmic}
\end{algorithm}

\begin{lemma}[Exactness of a topological sweep]
\label{lem:analysis_topological_sweep_exactness}
Assume that $(\mathcal{T}_h,\vec b)$ is sweepable and that every local
diagonal block $A_K$ is nonsingular. Let
$K_1,\dots,K_{N_{\mathrm{cell}}}$ be any topological ordering of
$\mathcal{G}_{h,\vec b}$. Then applying
\eqref{eq:analysis_cell_block_equation} successively for
$K_1,\dots,K_{N_{\mathrm{cell}}}$ computes the exact solution of the
assembled fixed-source DG system, up to floating-point and local solve
errors.
\end{lemma}

\begin{proof}
In a topological ordering, every upstream neighbour of $K_j$ lies among
$K_1,\dots,K_{j-1}$. The global matrix ordered by cells is therefore
block lower triangular with diagonal blocks $A_{K_j}$. Since these
blocks are nonsingular, block forward substitution is well-defined and
produces the unique solution. The update
\eqref{eq:analysis_cell_block_equation} is exactly this forward
substitution.
\end{proof}

\begin{corollary}[Batched sweep exactness]
\label{cor:analysis_batched_sweep_exactness}
Let $\mathcal I_\alpha^{(s)}$ be a graph-compatible sweep class whose
common dependency graph is acyclic, and let
\[
K_1,\dots,K_{N_{\mathrm{cell}}}
\]
be any topological ordering of that graph. Assume that
$A_{\alpha,K}$ is nonsingular for every
$\alpha\in\mathcal I_\alpha^{(s)}$ and every cell $K$.
Then Algorithm~\ref{alg:batched_wavefront_sweep}, restricted to
$\mathcal I_\alpha^{(s)}$, computes for every
$\alpha\in\mathcal I_\alpha^{(s)}$ the unique solution of the
corresponding independently assembled upwind DG system.

Consequently, applying the algorithm class by class gives the same
collection of discrete solutions as independent classical DG sweeps,
up to floating-point and local solve errors.
\end{corollary}

\begin{proof}
Graph compatibility gives every
$\alpha\in\mathcal I_\alpha^{(s)}$ the same directed dependency graph
and hence the same topological cell ordering. Fix one batch entry
$\alpha$. The $\alpha$-slice of
Algorithm~\ref{alg:batched_wavefront_sweep} applies
\eqref{eq:analysis_batched_local_update} successively in that ordering,
and no update for this slice depends on any other batch entry.
It is therefore exactly the block forward substitution considered in
Lemma~\ref{lem:analysis_topological_sweep_exactness}.
Since $\alpha$ is arbitrary, the result holds simultaneously for every
channel in the sweep class, and hence class by class for the full
partition.
\end{proof}

For Cartesian meshes and constant directions, the topological ordering
has an explicit wavefront form. Let the mesh have $N_q$ cells in
coordinate direction $q$. Define
\begin{equation}
s_q=\operatorname{sign}(b_q),
\qquad
A(s)=\qc{q:s_q\neq0}.
\end{equation}
For $q\in A(s)$, set
\begin{equation}
\iota_q(i_q)
=
\begin{cases}
i_q, & b_q>0,\\
N_q+1-i_q, & b_q<0.
\end{cases}
\label{eq:analysis_directional_index}
\end{equation}
The Cartesian wavefronts are
\begin{equation}
\mathcal{W}^{\vec b}_k
=
\qc{
\vec i:
\sum_{q\in A(s)}
\iota_q(i_q)
=
k
},
\qquad
\norm{A(s)}\leq k\leq \sum_{q\in A(s)}N_q.
\label{eq:analysis_wavefronts}
\end{equation}
Coordinates with $s_q=0$ are inactive in this ordering because the
corresponding faces are grazing.

\begin{corollary}[Cartesian wavefront exactness]
\label{cor:analysis_cartesian_wavefront_exactness}
Assume that $\vec b$ is constant and nonzero on a Cartesian mesh. Then
the wavefronts $\qc{\mathcal{W}^{\vec b}_k}$ form a topological ordering
of $\mathcal{G}_{h,\vec b}$. Consequently, one wavefront sweep computes
the exact solution of the fixed-source upwind DG system, up to
floating-point and local solve errors.
\end{corollary}

\begin{proof}
If a cell $\vec i$ depends on its upwind neighbour
$\vec i-s_q\vec e_q$, then $q\in A(s)$ and
$\iota_q(i_q-s_q)=\iota_q(i_q)-1$. All other active directional
indices are unchanged. Thus every graph
edge points from a smaller value of
$\sum_{q\in A(s)}\iota_q(i_q)$ to a larger value. The wavefront ordering
is topological, and the result follows from
Lemma~\ref{lem:analysis_topological_sweep_exactness}.
\end{proof}

In coupled scattering or fission iterations,
Corollary~\ref{cor:analysis_batched_sweep_exactness} applies at each
fixed-source stage after the coupled source has been formed.

\subsection{Cost model and batched execution}
\label{subsec:analysis_cost_model}

The preceding results show that the batched wavefront algorithm
computes the same block triangular DG solution as a classical sweep.
The purpose of batching is to change how many independent cell-local
solves are presented to the hardware at once; it does not reduce the
arithmetic work of a single fixed-source solve. Here we record the
resulting work, depth and storage scalings.

Let
\begin{equation}
N_{\mathrm{cell}}
=
\prod_{q=1}^d N_q,
\qquad
n_p=(p+1)^d,
\qquad
n_\alpha
=
N_{\mathrm{samp}}
N_{\mathrm{angle}}
N_{\mathrm{group}}
N_{\mathrm{rhs}}.
\end{equation}
Here $n_\alpha$ is the number of fixed-source channels before
partitioning into sweep classes; if several sweep classes are present,
then $n_\alpha=\sum_{s\in\mathcal{S}} n_\alpha^{(s)}$.

For one fixed-source channel, each cell update involves a bounded number
of upwind-neighbour contributions and one dense local solve of size
$n_p$.  With reused local factors or inverses, the work per cell is
$O(n_p^2)$ and one deterministic sweep costs
$O(N_{\mathrm{cell}}n_p^2)$. If the local block is factorised during
the sweep, this becomes
$O(N_{\mathrm{cell}}n_p^3)$.  For $n_\alpha$ independent channels the
corresponding arithmetic work is multiplied by $n_\alpha$.
Tensorisation preserves this count; it presents the independent local
DG block operations as batched dense tensor operations rather than as
an outer loop over scalar sweeps. This use of batched small-system
kernels is consistent with the wider accelerator linear-algebra
literature \cite{HaidarEtAl2015BatchedGPU,
  AbdelfattahEtAl2016BatchedGEMM} and with recent batched transport
kernels \cite{SuauEtAl2025Batched}.

For a sweep class $\mathcal{I}_\alpha^{(s)}$, the measured loop time
is the time obtained by advancing its channels one at a time, while
the batched time is the time obtained by advancing them together with
one common wavefront schedule.  The speedup we report in \S
\ref{sec:numerics} is
\begin{equation}
  S_{\mathrm{batch}}
  =
  \frac{
    \sum_{s\in\mathcal{S}} T_{\mathrm{loop}}^{(s)}
  }{
    \sum_{s\in\mathcal{S}} T_{\mathrm{batch}}^{(s)}
  }.
  \label{eq:analysis_batch_speedup}
\end{equation}
This is an execution-layout speedup for the stated sweep kernel.  It
depends on the available batch dimension, memory bandwidth, kernel
fusion and the efficiency of the batched dense linear algebra backend.

On a Cartesian mesh, the sequential depth of one sweep class is the
number of active wavefronts, $O\qp{\sum_{q\in A(s)} N_q}$.
Within a wavefront, the independent work is indexed by
$(K,\alpha)\in\mathcal{W}_k^{(s)}\times\mathcal{I}_\alpha^{(s)}.$ Thus
parallelism comes both from the cells on the current wavefront and
from the fixed-source channels in the current sweep class.

The price of exposing this parallelism is storage.  For one sweep
class, the solution tensor requires
$O(n_\alpha^{(s)}N_{\mathrm{cell}}n_p)$ entries.  If channel- and
cell-dependent local inverse or factor data are also stored, the
corresponding storage is
$O(n_\alpha^{(s)}N_{\mathrm{cell}}n_p^2).$

In memory-limited calculations the sample dimension can be split into
microbatches of size $B$, reducing the peak storage to
$O(BN_{\mathrm{cell}}n_p)$ for the solution data,
$O(BN_{\mathrm{cell}}n_p^2)$ when channel-dependent local factors are
stored, without changing the total arithmetic work over all samples.

The same storage issue appears when differentiating through the
unrolled wavefront program.  Reverse-mode automatic differentiation
has the same forward arithmetic cost, but as we will see the reverse
pass must store or reconstruct the primal states needed by the
transposed local solves.  Without checkpointing this requires
$O(n_\alpha N_{\mathrm{cell}}n_p)$ forward solution entries, in
addition to any stored local factor data; checkpointing over
wavefronts or sample microbatches trades this storage for
recomputation.

\begin{remark}
The prototype precomputes the small cell-local inverses and applies
them as batched matrix--vector products. The wavefront formulation
requires only the action of the local solve, so factor-based batched
solves can be substituted without changing the graph or batching
structure and may be preferable for poorly conditioned local blocks.
\end{remark}

\section{Differentiating the discrete wavefront algorithm}
\label{sec:uq}

This section records the finite-dimensional sensitivity structure of
the batched sweep. We focus on the discrete algorithm in that on any
parameter region where the upwind face signs and wavefront schedule
are fixed, the sweep is a differentiable composition of local matrix
constructions, tensor gathers and dense local solves. Reverse-mode
automatic differentiation through this unrolled program gives the same
algebraic adjoint action as differentiating the corresponding block
triangular DG system. When a parameter crosses a face-sign change, the
map is only piecewise smooth, and automatic differentiation returns
the derivative of the active branch.

The finite-dimensional differentiation used here follows the standard
connection between reverse-mode algorithmic differentiation, implicit
linear solves and discrete adjoints \cite{BaydinEtAl2018ADSurvey,
  Giles2008MatrixDerivatives, FarrellEtAl2013DolfinAdjoint,
  BlondelEtAl2022ImplicitDiff}.

\subsection{Discrete solution map and sweep branches}

Let $\xi$ be the finite-dimensional parameter vector describing
coefficients, inflow data, sources, cross sections or streaming
directions. For fixed mesh, angular quadrature and polynomial degree,
the upwind DG discretisation gives a finite-dimensional linear system
\begin{equation}
  L_h(\xi)U_h(\xi)=F_h(\xi),
  \label{eq:uq_discrete_system}
\end{equation}
where $U_h(\xi)\in\mathbb{R}^{N_h}$ collects the DG degrees of freedom
of the fixed-source problem.

The dependence on a streaming direction enters through the upwind split.
For a mesh face $F$ adjacent to a cell $K$, set
\begin{equation}
  b_F(\xi)=\vec b(\xi)\cdot\vec n_K .
\end{equation}
On an interior face $F=K\cap K'$, the contribution to the row associated
with a test function on $K$ has the form
\begin{equation}
  \int_F b_F^+(\xi) u_K v_K ds
  -
  \int_F b_F^-(\xi) u_{K'} v_K ds,
  \label{eq:uq_upwind_split_face}
\end{equation}
where
\begin{equation}
  b_F^+(\xi)=\max\qc{b_F(\xi),0},
  \qquad
  b_F^-(\xi)=\max\qc{-b_F(\xi),0}.
  \label{eq:uq_positive_negative_parts}
\end{equation}
Boundary faces have the analogous form, with $b_F^-(\xi)$ multiplying
the prescribed inflow data. If $b_F(\xi)=0$, the face is grazing and
contributes no normal flux.

\begin{proposition}[Continuity and branchwise smoothness]
\label{prop:uq_continuity_branchwise}
Assume that the assembled entries of $L_h(\xi)$ and $F_h(\xi)$ are
continuous in $\xi$ on a parameter neighbourhood $\mathcal{N}$, and
that $L_h(\xi)$ is nonsingular for all $\xi\in\mathcal{N}$. Then
\begin{equation}
  \xi\mapsto U_h(\xi)=L_h(\xi)^{-1}F_h(\xi)
\end{equation}
is continuous on $\mathcal{N}$.

If the assembled entries are locally Lipschitz in $\xi$, then
$\xi\mapsto U_h(\xi)$ is locally Lipschitz on $\mathcal{N}$. If, on an
open subset of $\mathcal{N}$, all nonzero face signs are fixed, all
structural grazing faces remain grazing, and the coefficient, source
and inflow parameterisations are $C^1$, then $\xi\mapsto U_h(\xi)$ is
$C^1$ on that subset.
\end{proposition}

\begin{proof}
The maps $b_F\mapsto b_F^+$ and $b_F\mapsto b_F^-$ are continuous and
locally Lipschitz. Hence the assembled entries inherit the corresponding
regularity from the input parameterisations. Matrix inversion is
continuous on the set of nonsingular matrices, so
$L_h(\xi)^{-1}F_h(\xi)$ is continuous.

For local Lipschitz continuity, use the identity
$A^{-1}-B^{-1}=A^{-1}(B-A)B^{-1}$ together with local boundedness of
$L_h(\xi)^{-1}$. On a fixed
face-sign branch, the positive and negative parts in
\eqref{eq:uq_positive_negative_parts} reduce either to smooth linear
functions of $b_F$ or to the identically zero function on structural
grazing faces. Thus $L_h(\xi)$ and $F_h(\xi)$ are $C^1$ on the branch,
and differentiability follows from the finite-dimensional
implicit-function theorem.
\end{proof}
This is the transport-sweep analogue of the standard discrete-adjoint
interpretation of reverse differentiation
\cite{FarrellEtAl2013DolfinAdjoint,Giles2008MatrixDerivatives}.

A \emph{face-sign branch} is a connected open parameter region on which
the sign of
$\vec b(\xi)\cdot\vec n_F$
is fixed for every interior and boundary face, with structural grazing
faces remaining identically grazing. A face-sign branch therefore fixes
the directed dependency graph, the boundary inflow pattern and every
coefficient selected by the positive and negative parts in
\eqref{eq:uq_positive_negative_parts}.

A \emph{fixed-schedule region} is a connected parameter region on
which one chosen topological ordering remains valid. Fixed face signs
are a sufficient, but not necessary, condition for a fixed schedule.
Different dependency graphs may share a common ordering, and
boundary-face signs do not create interior graph edges. The smoothness
and exact discrete-adjoint statements below are therefore made on
face-sign branches. On a face-sign branch with sign pattern $\sigma$,
write
\begin{equation}
L_h^\sigma(\xi)U_h^\sigma(\xi)
=
F_h^\sigma(\xi).
\label{eq:uq_branch_discrete_system}
\end{equation}
If $L_h^\sigma$ and $F_h^\sigma$ are $C^1$ and $L_h^\sigma$ is
nonsingular, then differentiating \eqref{eq:uq_discrete_system} gives
\begin{equation}
  L_h(\xi)
  \frac{\partial U_h}{\partial \xi_r}
  =
  \frac{\partial F_h}{\partial \xi_r}
  -
  \frac{\partial L_h}{\partial \xi_r}U_h(\xi).
  \label{eq:uq_forward_sensitivity_system}
\end{equation}
Thus
\begin{equation}
  \frac{\partial U_h}{\partial \xi_r}
  =
  L_h(\xi)^{-1}
  \qb{
    \frac{\partial F_h}{\partial \xi_r}
    -
    \frac{\partial L_h}{\partial \xi_r}U_h(\xi)
  }.
  \label{eq:uq_forward_sensitivity}
\end{equation}

For a differentiable scalar discrete observable
$J_h(U_h(\xi),\xi)$, the pathwise derivative can be written in adjoint
form. Define $\lambda_h(\xi)$ by
\begin{equation}
  L_h(\xi)^T\lambda_h(\xi)
  =
  \qp{
    \frac{\partial J_h}{\partial U_h}
  }^T.
  \label{eq:uq_discrete_adjoint}
\end{equation}
Then
\begin{equation}
  \frac{\partial}{\partial \xi_r}
  J_h(U_h(\xi),\xi)
  =
  \frac{\partial J_h}{\partial \xi_r}
  +
  \lambda_h(\xi)^T
  \qb{
    \frac{\partial F_h}{\partial \xi_r}
    -
    \frac{\partial L_h}{\partial \xi_r}U_h(\xi)
  }.
  \label{eq:uq_reverse_sensitivity}
\end{equation}

\subsection{Automatic differentiation through one sweep}

On one sweep branch, the tensorised wavefront sweep is a composition of
finite-dimensional operations with a fixed graph and fixed wavefront
schedule. For a batch index $\alpha=(n,m,\gamma,r)$ and cell $K$, the
local update has the form
\begin{equation}
  U_{\alpha,K}
  =
  A_{\alpha,K}(\xi)^{-1}
  b_{\alpha,K}(\xi),
  \label{eq:uq_local_solve}
\end{equation}
where
\begin{equation}
  b_{\alpha,K}(\xi)
  =
  R_{\alpha,K}(\xi)
  +
  \sum_{K'\in\mathcal{U}_{\alpha}(K)}
  B_{\alpha,K,K'}(\xi)U_{\alpha,K'}(\xi)
  +
  G_{\alpha,K}(\xi).
  \label{eq:uq_local_rhs}
\end{equation}
For a perturbation $\delta\xi$ inside the branch,
\begin{equation}
  \delta U_{\alpha,K}
  =
  A_{\alpha,K}^{-1}
  \qb{
    \delta b_{\alpha,K}
    -
    (\delta A_{\alpha,K})U_{\alpha,K}
  }.
  \label{eq:uq_local_solve_derivative}
\end{equation}
Thus differentiating a local dense solve is exactly differentiating the
corresponding local DG block equation.

\begin{proposition}[Reverse-mode derivative on a fixed face-sign branch]
\label{prop:uq_reverse_mode_wavefront_adjoint}
Consider one fixed face-sign branch with fixed mesh topology, fixed
upwind graph and a fixed chosen topological cell ordering. Assume that
all local blocks, coupling blocks, source vectors and inflow vectors
are continuously differentiable functions of $\xi$ on this branch, and
that every local diagonal block is nonsingular there. Let
\begin{equation}
  L_h(\xi)U_h(\xi)=F_h(\xi)
  \label{eq:uq_global_triangular_system}
\end{equation}
denote the assembled upwind DG system ordered by the sweep, and let
$\mathcal{A}_h(\xi)$ denote the wavefront algorithm. Then
\begin{equation}
  \mathcal{A}_h(\xi)=U_h(\xi)=L_h(\xi)^{-1}F_h(\xi).
\end{equation}
Moreover, for any differentiable scalar discrete observable
$J_h(U_h(\xi),\xi)$, reverse-mode differentiation through the unrolled
wavefront algorithm gives
\begin{equation}
  D_\xi J_h(U_h(\xi),\xi)[\delta\xi]
  =
  \partial_\xi J_h(U_h(\xi),\xi)[\delta\xi]
  +
  \lambda_h(\xi)^T
  \qb{
    D_\xi F_h(\xi)[\delta\xi]
    -
    D_\xi L_h(\xi)[\delta\xi]U_h(\xi)
  },
  \label{eq:uq_wavefront_reverse_directional}
\end{equation}
where $\lambda_h(\xi)$ solves
\begin{equation}
  L_h(\xi)^T\lambda_h(\xi)
  =
  \qp{
    \frac{\partial J_h}{\partial U_h}(U_h(\xi),\xi)
  }^T.
  \label{eq:uq_wavefront_adjoint_system}
\end{equation}
For batched computations with several sweep classes, the identity holds
on each class and scalar observable derivatives are obtained by summing
the corresponding class contributions.
\end{proposition}

\begin{proof}
On a fixed branch, the wavefront algorithm is block forward
substitution for the ordered triangular system
\eqref{eq:uq_global_triangular_system}. Hence
$\mathcal{A}_h(\xi)=L_h(\xi)^{-1}F_h(\xi)$.

Differentiating \eqref{eq:uq_global_triangular_system} in direction
$\delta\xi$ gives
\begin{equation}
  L_h(\xi)\delta U_h
  =
  D_\xi F_h(\xi)[\delta\xi]
  -
  D_\xi L_h(\xi)[\delta\xi]U_h(\xi).
  \label{eq:uq_wavefront_forward_directional}
\end{equation}
Combining this with the observable derivative and the adjoint equation
\eqref{eq:uq_wavefront_adjoint_system} gives
\eqref{eq:uq_wavefront_reverse_directional}. Reverse-mode automatic
differentiation traverses the local operations in reverse topological
order. The transpose of each local solve is the transpose local block
solve associated with \eqref{eq:uq_local_solve_derivative}; the reverse
pass is therefore block backward substitution with $L_h(\xi)^T$.
\end{proof}

Let
\begin{equation}
  \mathcal{J}_h(\xi)
  =
  J_h(\mathcal{A}_h(\xi),\xi).
\end{equation}
Automatic differentiation through a fixed wavefront program computes
\begin{equation}
  D^{\mathrm{AD}}\mathcal{J}_h(\xi)
  =
  \nabla_\xi J_h(U_h(\xi),\xi),
  \label{eq:uq_ad_derivative}
\end{equation}
where the derivative is the derivative of the active sweep branch.

\subsection{Face-sign changes}

At a non-structural grazing configuration,
\[
  \vec b(\xi)\cdot\vec n_F=0,
\]
the positive- and negative-part functions entering the upwind flux are
locally Lipschitz but are generally not classically differentiable.
Consequently, although the discrete solution map
$\xi\mapsto U_h(\xi)$ remains locally Lipschitz under the assumptions of
Proposition~\ref{prop:uq_continuity_branchwise}, its derivative need not
be uniquely defined at such a parameter value.

On every open fixed-face-sign branch, however, the discrete solution
map is $C^1$ and the reverse-mode derivative agrees exactly with the
corresponding discrete adjoint.  As a grazing configuration is
approached from different neighbouring branches, the branchwise
derivatives need not agree.  Moreover, automatic differentiation
evaluated exactly at the grazing configuration depends on the
derivative convention assigned to the active nonsmooth primitive
\cite{LeeEtAl2020NonsmoothAD}.

We therefore make no branch-independent differentiability claim at
grazing configurations; the exact discrete-adjoint identity is asserted
only on open fixed-face-sign branches.

If the parameter law $\mu$ is absolutely continuous with respect to
Lebesgue measure, local Lipschitz continuity implies that the discrete
solution map is differentiable $\mu$-almost everywhere.

\subsection{Sampling and derivative-based summaries}

For a smooth discrete observable $J_h(U_h(\xi),\xi)$, pathwise
derivatives give derivative-based sensitivity measures of the type
commonly used in global sensitivity analysis
\cite{SobolKucherenko2009DGSM,Cacuci2003}. Define
\begin{equation}
  \nu_r(J_h)
  =
  \mathbb{E}
  \qb{
    \norm{
      \frac{\partial}{\partial \xi_r}
      J_h(U_h(\xi),\xi)
    }^2
  },
  \label{eq:uq_dgsm_raw}
\end{equation}
with Monte Carlo approximation
\begin{equation}
  \widehat{\nu}_{r,N}(J_h)
  =
  \frac{1}{N}
  \sum_{n=1}^{N}
  \norm{
    \frac{\partial}{\partial \xi_r}
    J_h(U_h(\xi^{(n)}),\xi^{(n)})
  }^2.
  \label{eq:uq_dgsm_estimator}
\end{equation}
The normalised derivative score is
\begin{equation}
  S_r^{\mathrm{DGSM}}(J_h)
  =
  \frac{
    \widehat{\nu}_{r,N}(J_h)
  }{
    \sum_{s=1}^{N_\xi}
    \widehat{\nu}_{s,N}(J_h)
  }.
  \label{eq:uq_dgsm_score}
\end{equation}
We use these scores to represent derivative-based rankings of the
local parameter influence on the chosen discrete observable. If the
input distribution assigns zero probability to face-sign-change sets,
the branch derivative exists almost surely under the local Lipschitz
assumptions above. Pathwise branch gradients are obtained by
differentiating the batched computational graph on the active sweep
branch and then reducing the samplewise derivatives in
\eqref{eq:uq_dgsm_estimator}--\eqref{eq:uq_dgsm_score}.

\section{Numerical experiments}
\label{sec:numerics}

The experiments test four aspects of the method. Example~1 verifies the
underlying upwind DG discretisation and the wavefront implementation on
a manufactured fixed-source problem. Example~2 isolates the effect of
batching independent fixed-source channels and measures the associated
memory-throughput tradeoff. Example~3 studies scalar material-shadowing
uncertainty and uses the same setting to test derivative-based
diagnostics and adjoint consistency. Example~4 embeds the batched sweep
inside a coupled multigroup power iteration and serves as a larger
algorithmic stress test.

Unless otherwise stated, the experiments are posed on
$D=(0,1)^2$ with tensor-product discontinuous finite elements on
uniform Cartesian meshes. Inflow data are imposed weakly through the
upwind numerical flux. Volume, face and diagnostic integrals are
evaluated using a high order Gauss--Legendre quadrature for the DG
terms and reported observables. In timing comparisons, the looped and
batched runs use the same discrete equations; only the execution
layout over independent fixed-source channels is changed.

The computations were run on a Dell Precision 5820 Tower with an Intel
Core i9-10900X CPU, $64$ GiB RAM and an NVIDIA RTX A4000 GPU. The GPU
timings used Python~3.14.4, PyTorch~2.11 and CUDA
runtime~12.8.  Timing data are reported as wall-clock
measurements. 

\ExampleSubsection{Verification of one fixed-source sweep}
{subsec:equivalence_convergence}

The first experiment verifies the wavefront implementation on a
manufactured fixed-source problem. The purpose is to check the
convergence of the underlying upwind DG discretisation, the weak
imposition of inflow data and the scaling of the sweep time with the
total number of batched degrees of freedom.

The domain is $D=(0,1)^2$. We use $N_{\mathrm{samp}}=64$ independent
realisations of
$\xi=(\xi_1,\xi_2,\xi_3,\xi_4)$, with independent components uniformly
distributed on $\qb{-1,1}$. For each realisation,
\begin{equation}
\vec b(\xi)
=
\begin{pmatrix}
\cos{\theta(\xi)}\\
\sin{\theta(\xi)}
\end{pmatrix},
\qquad
\theta(\xi)=\frac{\pi}{5}+\frac{3\pi}{25}\xi_1.
\label{eq:ex1_transport_field}
\end{equation}
Thus $\theta(\xi)\in\qb{2\pi/25,8\pi/25}$, so all samples lie in the same
sweep branch and the sweep proceeds from the south-west corner to the
north-east corner. The reaction coefficient is
\begin{equation}
c(x,y,\xi)
=
c_0(\xi)+c_x(\xi)x+c_y(\xi)y,
\label{eq:ex1_reaction}
\end{equation}
where
\begin{equation}
c_0(\xi)=1+\tfrac{3}{20}\xi_2,
\qquad
c_x(\xi)=\tfrac{7}{20}+\tfrac{2}{25}\xi_3,
\qquad
c_y(\xi)=\tfrac{1}{4}+\tfrac{2}{25}\xi_4.
\end{equation}
These choices give $c(x,y,\xi)\geq\tfrac{17}{20}$ for all
$(x,y)\in D$ and all sampled inputs.

The manufactured exact solution is
\begin{equation}
u_{\mathrm{ex}}(x,y,\xi)
=
1
+
a(\xi)
\sin{\qp{\pi x+\phi_x(\xi)}}
\cos{\qp{\tfrac{3\pi}{4}y+\phi_y(\xi)}}
+
\tau(\xi)x(1-y),
\label{eq:ex1_exact_solution}
\end{equation}
with
\begin{equation}
a(\xi)=\tfrac{1}{5}+\tfrac{1}{20}\xi_3,
\qquad
\phi_x(\xi)=\tfrac{1}{4}\xi_1,
\qquad
\phi_y(\xi)=-\tfrac{1}{5}\xi_2,
\qquad
\tau(\xi)=\tfrac{1}{10}+\tfrac{1}{25}\xi_4.
\end{equation}
The source and inflow data are defined by
\begin{equation}
f(x,y,\xi)
=
\vec b(\xi)\cdot \nabla u_{\mathrm{ex}}(x,y,\xi)
+
c(x,y,\xi)u_{\mathrm{ex}}(x,y,\xi),
\label{eq:ex1_source}
\end{equation}
and
\begin{equation}
g(x,y,\xi)=u_{\mathrm{ex}}(x,y,\xi)
\quad \text{on } \Gamma_-(\vec b(\xi)).
\label{eq:ex1_inflow}
\end{equation}

The computation uses $\mathbb{Q}_p$ elements with $p=0,1,2$ on uniform meshes
with $N_x=N_y=N$. The field errors are measured samplewise and then
averaged:
\begin{align}
\overline{e}_{L^2}
&=
\frac{1}{N_{\mathrm{samp}}}
\sum_{n=1}^{N_{\mathrm{samp}}}
\Norm{
u_h(\cdot,\xi^{(n)})
-
u_{\mathrm{ex}}(\cdot,\xi^{(n)})
}_{L^2(D)},
\label{eq:ex1_mean_l2_error}
\\
\overline{e}_{\mathrm{DG}}
&=
\frac{1}{N_{\mathrm{samp}}}
\sum_{n=1}^{N_{\mathrm{samp}}}
\Norm{
u_h(\cdot,\xi^{(n)})
-
u_{\mathrm{ex}}(\cdot,\xi^{(n)})
}_{\mathrm{DG},\xi^{(n)}}.
\label{eq:ex1_mean_dg_error}
\end{align}
The detector current on the right boundary is
\begin{equation}
J_{\mathrm{det}}(u;\xi)
=
\int_{7/20}^{13/20}
b_x(\xi)u(1,y,\xi)dy,
\label{eq:ex1_detector}
\end{equation}
and the mean absolute detector-current error is denoted
$\overline{e}_{\mathrm{det}}$. Observed convergence rates are computed
by
\begin{equation}
r_N(e)
=
\frac{\log(e_N/e_{2N})}{\log 2}.
\label{eq:observed_rate}
\end{equation}
The reported sweep time $t_{\mathrm{sweep}}$ is the median time for
the wavefront forward substitution over the full batch; it excludes
local block construction and inversion. The total number of batched DG
degrees of freedom is
\begin{equation}
  N_{\mathrm{dof}}
  =
  N_{\mathrm{samp}}N_xN_y(p+1)^2.
  \label{eq:ex1_total_dofs}
\end{equation}

\begin{table}[h!]
\centering
\scriptsize
\setlength{\tabcolsep}{3.5pt}
\begin{tabular}{cccccccccc}
\hline
$p$ & $N_x$ & $N_{dof}$ & $\overline{e}_{L^2}$ & $r_N(\overline{e}_{L^2})$ & $\overline{e}_{\mathrm{DG}}$ & $r_N(\overline{e}_{\mathrm{DG}})$ & $\overline{e}_{\mathrm{det}}$ & $r_N(\overline{e}_{\mathrm{det}})$ & sweep (s) \\
\hline
0 & 8 & $4,096$ & $2.207\times 10^{-2}$ & -- & $8.126\times 10^{-2}$ & -- & $2.355\times 10^{-3}$ & -- & 0.013 \\
0 & 16 & $16,384$ & $1.184\times 10^{-2}$ & 0.90 & $5.865\times 10^{-2}$ & 0.47 & $1.376\times 10^{-3}$ & 0.78 & 0.026 \\
0 & 32 & $65,536$ & $6.205\times 10^{-3}$ & 0.93 & $4.194\times 10^{-2}$ & 0.48 & $7.487\times 10^{-4}$ & 0.88 & 0.053 \\
0 & 64 & $262,144$ & $3.195\times 10^{-3}$ & 0.96 & $2.983\times 10^{-2}$ & 0.49 & $3.902\times 10^{-4}$ & 0.94 & 0.108 \\
0 & 128 & $1,048,576$ & $1.626\times 10^{-3}$ & 0.97 & $2.116\times 10^{-2}$ & 0.50 & $1.992\times 10^{-4}$ & 0.97 & 0.216 \\
0 & 256 & $4,194,304$ & $8.213\times 10^{-4}$ & 0.99 & $1.499\times 10^{-2}$ & 0.50 & $1.006\times 10^{-4}$ & 0.99 & 0.437 \\
\hline
1 & 8 & $16,384$ & $9.142\times 10^{-4}$ & -- & $4.342\times 10^{-3}$ & -- & $1.103\times 10^{-5}$ & -- & 0.013 \\
1 & 16 & $65,536$ & $2.329\times 10^{-4}$ & 1.97 & $1.549\times 10^{-3}$ & 1.49 & $1.546\times 10^{-6}$ & 2.83 & 0.028 \\
1 & 32 & $262,144$ & $5.876\times 10^{-5}$ & 1.99 & $5.498\times 10^{-4}$ & 1.49 & $2.010\times 10^{-7}$ & 2.94 & 0.058 \\
1 & 64 & $1,048,576$ & $1.476\times 10^{-5}$ & 1.99 & $1.947\times 10^{-4}$ & 1.50 & $2.593\times 10^{-8}$ & 2.95 & 0.116 \\
1 & 128 & $4,194,304$ & $3.699\times 10^{-6}$ & 2.00 & $6.890\times 10^{-5}$ & 1.50 & $2.839\times 10^{-9}$ & 3.19 & 0.236 \\
1 & 256 & $16,777,216$ & $9.257\times 10^{-7}$ & 2.00 & $2.437\times 10^{-5}$ & 1.50 & $3.819\times 10^{-10}$ & 2.89 & 0.500 \\
\hline
2 & 8 & $36,864$ & $2.899\times 10^{-5}$ & -- & $1.694\times 10^{-4}$ & -- & $7.386\times 10^{-8}$ & -- & 0.015 \\
2 & 16 & $147,456$ & $3.650\times 10^{-6}$ & 2.99 & $3.008\times 10^{-5}$ & 2.49 & $1.703\times 10^{-9}$ & 5.44 & 0.030 \\
2 & 32 & $589,824$ & $4.577\times 10^{-7}$ & 3.00 & $5.328\times 10^{-6}$ & 2.50 & $3.680\times 10^{-10}$ & 2.21 & 0.063 \\
2 & 64 & $2,359,296$ & $5.731\times 10^{-8}$ & 3.00 & $9.426\times 10^{-7}$ & 2.50 & $1.318\times 10^{-11}$ & 4.80 & 0.125 \\
2 & 128 & $9,437,184$ & $7.170\times 10^{-9}$ & 3.00 & $1.667\times 10^{-7}$ & 2.50 & $1.343\times 10^{-12}$ & 3.29 & 0.277 \\
2 & 256 & $37,748,736$ & $8.966\times 10^{-10}$ & 3.00 & $2.947\times 10^{-8}$ & 2.50 & $3.369\times 10^{-14}$ & 5.32 & 0.690 \\
\hline
\end{tabular}
\caption{\CurrentExample: Manufactured random-coefficient fixed-source sweep verification. Mean $L^2(D)$ errors, DG-norm errors, detector-current errors, observed rates and median wavefront sweep timings for $Q_0$, $Q_1$ and $Q_2$ upwind DG. The reported degrees of freedom include all samples.}
\label{tab:ex1_convergence}
\end{table}

\begin{figure}[h!]
\centering

\begin{subfigure}[t]{0.48\textwidth}
\centering
\includegraphics[width=\textwidth]{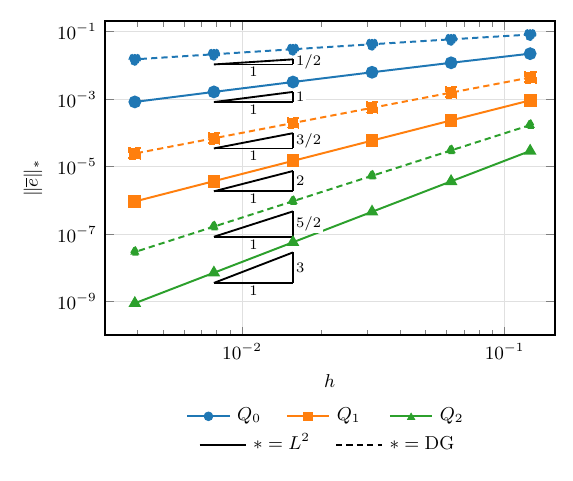}
\caption{Mean field errors.}
\label{fig:ex1_l2_dg_convergence}
\end{subfigure}
\hfill
\begin{subfigure}[t]{0.48\textwidth}
\centering
\includegraphics[width=\textwidth]{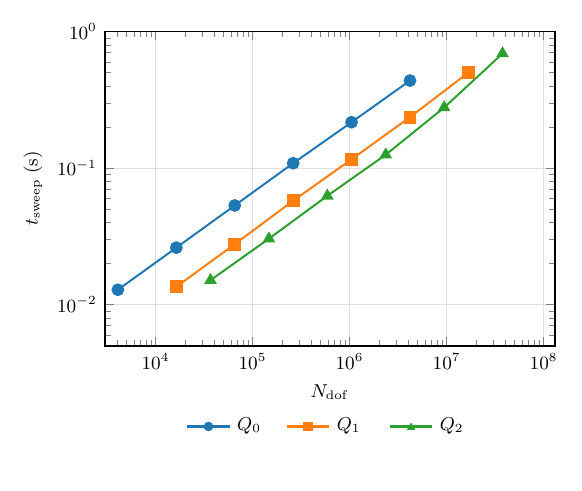}
\caption{Median sweep time.}
\label{fig:ex1_runtime_dofs}
\end{subfigure}

\caption{\CurrentExample: manufactured random-coefficient fixed-source
  sweep verification. Panel~\textup{(a)} shows the mean samplewise
  $L^2(D)$ and DG-norm errors under mesh refinement for
  $\mathbb{Q}_0$, $\mathbb{Q}_1$ and $\mathbb{Q}_2$ upwind
  DG. Panel~\textup{(b)} shows the median wavefront sweep time
  $t_{\mathrm{sweep}}$ as a function of the total number of batched DG
  degrees of freedom $N_{\mathrm{dof}}$.}
\label{fig:ex1_verification}
\end{figure}

Table~\ref{tab:ex1_convergence} and Figure~\ref{fig:ex1_verification}
verify the expected discretisation behaviour. The mean $L^2(D)$ and
DG-norm errors converge at approximately orders $p+1$ and
$p+\tfrac12$, respectively, consistent with the classical upwind DG
behaviour on this Cartesian setting
\cite{JohnsonPitkaranta1986,CockburnDongGuzman2008}. The
detector-current error supplies an additional scalar-output check and
converges more rapidly in this manufactured example; the finest
$\mathbb Q_2$ values are at floating-point scale. Over the tested
range, the measured sweep time grows sublinearly with
$N_{\mathrm{dof}}$, with an empirical trend close to
$t_{\mathrm{sweep}}\propto N_{\mathrm{dof}}^{1/2}$, reflecting the
increasing concurrency available within the wavefronts. The arithmetic
work scaling is given in Section~\ref{subsec:analysis_cost_model}.

\ExampleSubsection{Performance and microbatch scaling}
{subsec:batching_scaling}

The second experiment isolates the performance effect of carrying
independent fixed-source solves as a leading tensor dimension. In every
looped and batched comparison, the spatial discretisation,
sample-dependent coefficients, cell-local equations and wavefront
ordering are identical; only the execution layout over the sample
index is changed.

The comparison is therefore closest in spirit to embedded-ensemble and
batched transport approaches, in which independent algebraic problems
are presented simultaneously to the hardware
\cite{PhippsEtAl2017EmbeddedEnsemble, DEliaEtAl2018EnsembleGrouping,
  SuauEtAl2025Batched}.

The baseline problem uses $\mathbb{Q}_1$ upwind DG on a uniform
$64\times64$ mesh of $D=(0,1)^2$, giving $16{,}384$ DG unknowns per
sample. We solve
\begin{equation}
\vec b(\xi)\cdot\nabla u(x,\xi)
+
c(x,\xi)u(x,\xi)
=
0,
\qquad x\in D,
\label{eq:ex2_transport_problem}
\end{equation}
with
\begin{equation}
\vec b(\xi)
=
\begin{pmatrix}
\cos{\theta(\xi)}\\
\sin{\theta(\xi)}
\end{pmatrix},
\qquad
\theta(\xi)\sim
\mathcal U\qp{\tfrac{3}{100},\tfrac{3}{10}}.
\label{eq:ex2_direction}
\end{equation}
All sampled directions have positive coordinate components and
therefore share the same south-west to north-east sweep schedule.

The left-boundary inflow is
\begin{equation}
g(0,y,\xi)
=
A(\xi)
\exp\qp{
-\frac{(y-y_0(\xi))^2}{2\sigma_b(\xi)^2}
},
\qquad
g(x,0,\xi)=0,
\label{eq:ex2_inflow}
\end{equation}
where
\begin{equation}
y_0(\xi)\sim
\mathcal U\qp{\tfrac{7}{20},\tfrac{13}{20}},
\qquad
\sigma_b(\xi)\sim
\mathcal U\qp{\tfrac{7}{200},\tfrac{9}{100}},
\qquad
A(\xi)=\exp\qp{\tfrac{Z}{5}},
\quad
Z\sim\mathcal N(0,1).
\label{eq:ex2_inflow_parameters}
\end{equation}
The reaction coefficient is
\begin{equation}
c(x,y,\xi)
=
c_{\mathrm{bg}}(\xi)
+
a_c(\xi)
\exp\qp{
-\frac{
625\qp{(x-\tfrac{31}{50})^2+(y-\tfrac12)^2}
}{18}
},
\label{eq:ex2_absorption}
\end{equation}
with
\begin{equation}
c_{\mathrm{bg}}(\xi)
\sim
\mathcal U\qp{\tfrac35,\tfrac75},
\qquad
a_c(\xi)\sim\mathcal U(0,3).
\label{eq:ex2_absorption_parameters}
\end{equation}
All inputs are independent. Hence the samples have distinct directions,
local inverse blocks and weak inflow vectors, while retaining one
common dependency graph and wavefront sequence.

For sweep-only timings, the channel-dependent local inverse blocks and
inflow vectors are prepared before the timed region, and the solution
and wavefront-workspace tensors are preallocated. A timed sweep includes
initialisation of these arrays, gathering the west and south upwind
states, applying the face couplings and stored local inverses, inserting
the weak inflow data and scattering the updated cell coefficients.

CUDA synchronisation is performed immediately before and after each
complete timed region, with no synchronisation inside the sample or
microbatch loops. Launch and sweep-only measurements use two untimed
warm-up runs and ten timed repetitions; complete-pipeline and
mesh-scaling measurements use one warm-up and five repetitions.
Medians are reported, with interquartile ranges where shown.

We compare direct CUDA launching with CUDA-graph replay. In the latter,
the full wavefront sequence is captured once and subsequently replayed.
For the scalar baseline, a static single-sample graph is replayed once
per sample, with the corresponding inverse blocks, inflow data and
direction coefficients copied into fixed buffers without intermediate
synchronisation. The batched layout instead replays one graph per
batch. Graph replay therefore reduces ordinary host-side dispatch in
both layouts, but it does not equalise the number of graph replays or
fixed-buffer updates. The resulting comparison measures the complete
execution-layout benefit, including replay count, buffer traffic and
sample-parallel device work.

\subsubsection{Launch amortisation.}

The launch study uses $N_{\mathrm{samp}} \in
\qc{1,2,4,8,16,32,64,128,256}.$ For each sample count, the same
prepared problems are processed either sequentially or as one tensor
batch.

\begin{figure}[h!]
\centering

\begin{subfigure}[t]{0.48\textwidth}
\centering
\includegraphics[width=\textwidth]
{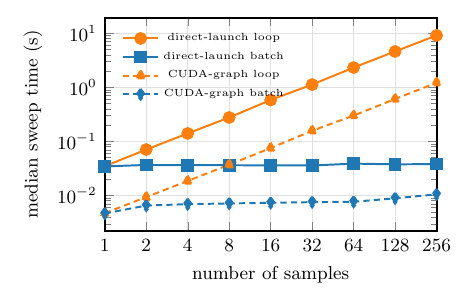}
\caption{Median scalar-loop and batched sweep times.}
\label{fig:ex2_launch_times}
\end{subfigure}
\hfill
\begin{subfigure}[t]{0.48\textwidth}
\centering
\includegraphics[width=\textwidth]
{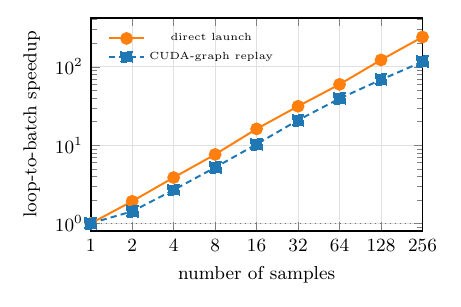}
\caption{Corresponding loop-to-batch speedups.}
\label{fig:ex2_launch_speedup}
\end{subfigure}

\caption{\CurrentExample: launch-amortisation study for identical
collections of fixed-source sample problems. Left: measured
scalar-loop and batched sweep times under direct CUDA launching and
CUDA-graph replay. Right: the corresponding loop-to-batch speedups.
Graph replay reduces ordinary dispatch overhead in both layouts.
Because the scalar layout still replays a graph and updates fixed
buffers once per sample, the remaining separation is an end-to-end
execution-layout speedup.}
\label{fig:ex2_launch_amortisation}
\end{figure}

\subsubsection{Memory--throughput tradeoff.}

The second study fixes the workload at $N_{\mathrm{samp}}=4096$ and
varies the sample microbatch size $B \in
\qc{1,2,4,8,16,32,64,128,256,512,1024,2048,4096}.$
The full $\mathbb{Q}_1$ workload contains $16{,}777{,}216$
sample--cell updates, or $67{,}108{,}864$
sample--cell--DG-unknown updates.

Two timings are reported. The \emph{prepared-sweep} measurement
constructs one representative batch of width $B$ and repeatedly applies
the sweep until the equivalent of $4096$ sample solves has been
completed. The \emph{complete-pipeline} measurement processes all
$4096$ distinct samples and, for every microbatch, constructs the
reaction fields and local blocks on the device, computes their inverses,
forms the weak inflow vectors, allocates the work arrays and performs
the sweep. 

\begin{figure}[h!]
\centering

\begin{subfigure}[t]{0.48\textwidth}
\centering
\includegraphics[width=\textwidth]
{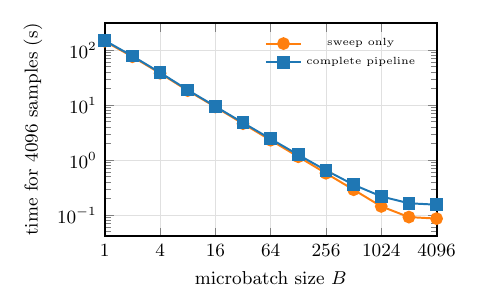}
\caption{Time for a fixed $4096$-sample-equivalent workload.}
\label{fig:ex2_microbatch_timing}
\end{subfigure}
\hfill
\begin{subfigure}[t]{0.48\textwidth}
\centering
\includegraphics[width=\textwidth]
{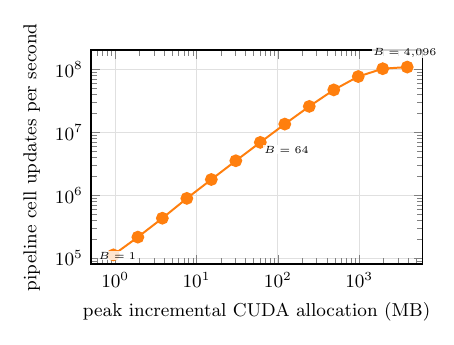}
\caption{Memory--throughput tradeoff.}
\label{fig:ex2_memory_throughput}
\end{subfigure}

\caption{\CurrentExample: sample-microbatch study. Left: median time
for a fixed $4096$-sample-equivalent prepared-sweep workload and for
the complete pipeline applied to all $4096$ distinct samples. Right:
complete-pipeline throughput against peak incremental CUDA allocation.
Increasing $B$ exposes more sample parallelism but increases the
resident channel-dependent storage.}
\label{fig:ex2_microbatch_tradeoff}
\end{figure}

Figure~\ref{fig:ex2_microbatch_tradeoff} shows the expected
memory--throughput tradeoff. Increasing $B$ reduces the number of
wavefront traversals and increases the independent work within each
stage, but the improvement diminishes near device saturation. The
reported memory is the incremental CUDA allocation relative to the
state immediately preceding preparation or execution. Because the
dominant inverse-block storage scales as
$\Oh(BN_xN_y(p+1)^4)$, an intermediate microbatch can recover most of
the available throughput with substantially less memory than the full
batch.

\subsubsection{Mesh, degree and batch scaling.}

The third study varies the spatial resolution, polynomial degree and
batch size according to $N_x=N_y=N \in \qc{32,64,128,256},
p\in\qc{0,1,2}, B\in\qc{1,64,512,4096}.$ One sweep contains $BN^2$
sample--cell updates and $BN^2(p+1)^2$ sample--cell--DG-unknown
updates. Its wavefront depth is $2N-1$, so the average number of
sample--cell updates exposed per causal stage is $\tfrac{BN^2}{2N-1}.$

For each configuration, we record direct and CUDA-graph sweep times,
cell-update and DG-unknown-update throughputs, wavefront depth,
prepared-batch storage and peak incremental CUDA allocation in Figure
\ref{fig:ex2_mesh_degree_scaling}.

\begin{figure}[h!]
\centering
\includegraphics[width=0.58\textwidth]
{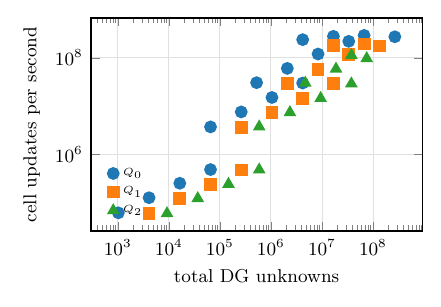}
\caption{\CurrentExample: mesh, degree and batch scaling for the
  CUDA-graph sweep. Throughput initially increases as more independent
  work is exposed within each wavefront and then approaches a
  degree-dependent saturated regime. Scatter between configurations
  with the same total number of DG unknowns reflects the separate
  effects of wavefront depth and batch width.}
\label{fig:ex2_mesh_degree_scaling}
\end{figure}

Figure~\ref{fig:ex2_mesh_degree_scaling} confirms that tensor size alone
does not determine throughput. For example, the $\mathbb Q_0$
configurations $(N,B)=(32,4096)$ and $(256,64)$ both contain
$4{,}194{,}304$ unknowns, but expose approximately
$6.66\times10^4$ and $8.21\times10^3$ cell updates per wavefront stage,
respectively; their throughputs are $3.060\times10^8$ and
$1.122\times10^8$ cell updates per second. Performance therefore
depends separately on wavefront depth and within-stage concurrency.
Graph replay is most beneficial in the launch-limited regime: the
direct-to-graph speedup is $4.61$--$9.03$ at $B=1$,
$2.82$--$6.53$ at $B=64$, and $1.01$--$1.29$ at $B=4096$.

\begin{table}[h!]
\centering
\begin{tabular}{ccccc}
\hline
degree
&
peak throughput
&
largest successful $(N,B)$
&
DG unknowns
&
peak allocation (MB)
\\
\hline
$\mathbb{Q}_0$
&
$3.060\times10^8$
&
$(256,4096)$
&
$2.684\times10^8$
&
$1.074\times10^4$
\\
$\mathbb{Q}_1$
&
$2.058\times10^8$
&
$(256,512)$
&
$1.342\times10^8$
&
$7.785\times10^3$
\\
$\mathbb{Q}_2$
&
$1.202\times10^8$
&
$(128,512)$
&
$7.550\times10^7$
&
$8.657\times10^3$
\\
\hline
\end{tabular}
\caption{Summary of the mesh--degree scaling study. The peak
  throughput is the largest measured CUDA-graph cell-update rate for
  each degree. The final three columns give the configuration with the
  largest total number of batched DG unknowns; memory is the peak
  incremental allocation during prepared-batch construction.}
\label{tab:ex2_mesh_degree_limits}
\end{table}

Table~\ref{tab:ex2_mesh_degree_limits} shows that the attainable batch
size decreases with polynomial degree because the stored local inverse
blocks scale quadratically with the cell-local dimension
$(p+1)^2$. The useful batch range is therefore determined jointly by
wavefront occupancy, polynomial degree and available device memory;
there is no hardware-independent optimal batch size.

As an implementation check, the looped, batched and CUDA-graph outputs
were compared for every launch-study sample count. Independent scalar
references were also constructed for the first $16$ samples in the
microbatch study. All comparisons agreed to the level expected from
single-precision roundoff, confirming that the measured performance
differences arise from execution layout rather than changes to the
discrete sweep.

\FloatBarrier

\ExampleSubsection{Material-shadowing UQ and adjoint verification}
{subsec:material_shadowing}

This experiment applies the batched sweep to a fixed beam passing
through an absorber with uncertain position, width and strength. A
representative realisation first illustrates the physical
mechanism. The absorber intersects and attenuates the beam, producing
a downstream shadow. Ensemble statistics then quantify the spatial
uncertainty in this shadow, while scalar responses connect the
four-dimensional input uncertainty to an effective beamline optical
depth.

\subsubsection{Problem, ensemble and representative realisation.}

We solve
\begin{align}
\vec b\cdot\nabla u(x,\zeta)+c(x,\zeta)u(x,\zeta)
&=0,
&&x\in D,
\label{eq:shadow_transport}
\\
\vec b
&=
\begin{pmatrix}
\cos\theta\\
\sin\theta
\end{pmatrix},
\qquad
\theta=\tfrac{7}{50}.
\label{eq:shadow_direction}
\end{align}
Both components of $\vec b$ are positive, so every sample uses the
south-west to north-east sweep. The inflow is
\begin{equation}
g(0,y)
=
\exp\qp{
-\frac{(y-y_b)^2}{2\sigma_b^2}
},
\qquad
y_b=\tfrac{23}{50},
\qquad
\sigma_b=\tfrac{3}{40},
\qquad
g(x,0)=0.
\label{eq:shadow_inflow}
\end{equation}

The independent standardised variables
\begin{equation}
\zeta=(\zeta_1,\zeta_2,\zeta_3,\zeta_4)
\sim\mathcal U\qp{\qb{-1,1}^4}
\label{eq:shadow_standardised_variables}
\end{equation}
determine
\begin{align}
x_c(\zeta)
&=\tfrac{13}{25}+\tfrac15\zeta_1,
&
y_c(\zeta)
&=\tfrac12+\tfrac{11}{50}\zeta_2,
\label{eq:shadow_standardised_centres}
\\
\sigma_c(\zeta)
&=\tfrac{7}{80}+\tfrac{17}{400}\zeta_3,
&
a_{\mathrm{inc}}(\zeta)
&=\tfrac{15}{2}+\tfrac92\zeta_4.
\label{eq:shadow_standardised_width_strength}
\end{align}
The absorption coefficient is
\begin{equation}
c(x,y,\zeta)
=
c_{\mathrm{bg}}
+
a_{\mathrm{inc}}(\zeta)
\exp\qp{
-\frac{(x-x_c(\zeta))^2+(y-y_c(\zeta))^2}
{2\sigma_c(\zeta)^2}
},
\qquad
c_{\mathrm{bg}}=\tfrac14.
\label{eq:shadow_absorption}
\end{equation}

The main ensemble uses $\mathbb Q_1$ DG on a $128\times128$ mesh with
$N_{\mathrm{samp}}=4096$. It therefore performs
\[
4096\times128^2
=
67{,}108{,}864
\]
sample--cell solves and
\[
4\times67{,}108{,}864
=
268{,}435{,}456
\]
sample--cell--unknown updates over $255$ wavefronts. The forward
calculation uses single precision, precomputed local inverse blocks
and in-place GPU updates.

Sensitivities are computed for $N_{\mathrm{DGSM}}=2048$ of the same
samples in differentiable microbatches of width $128$. These
microbatches use batched local solves and avoid in-place modifications
of tensors required by automatic differentiation.

\begin{figure}[h!]
\centering
\includegraphics[width=\textwidth]
{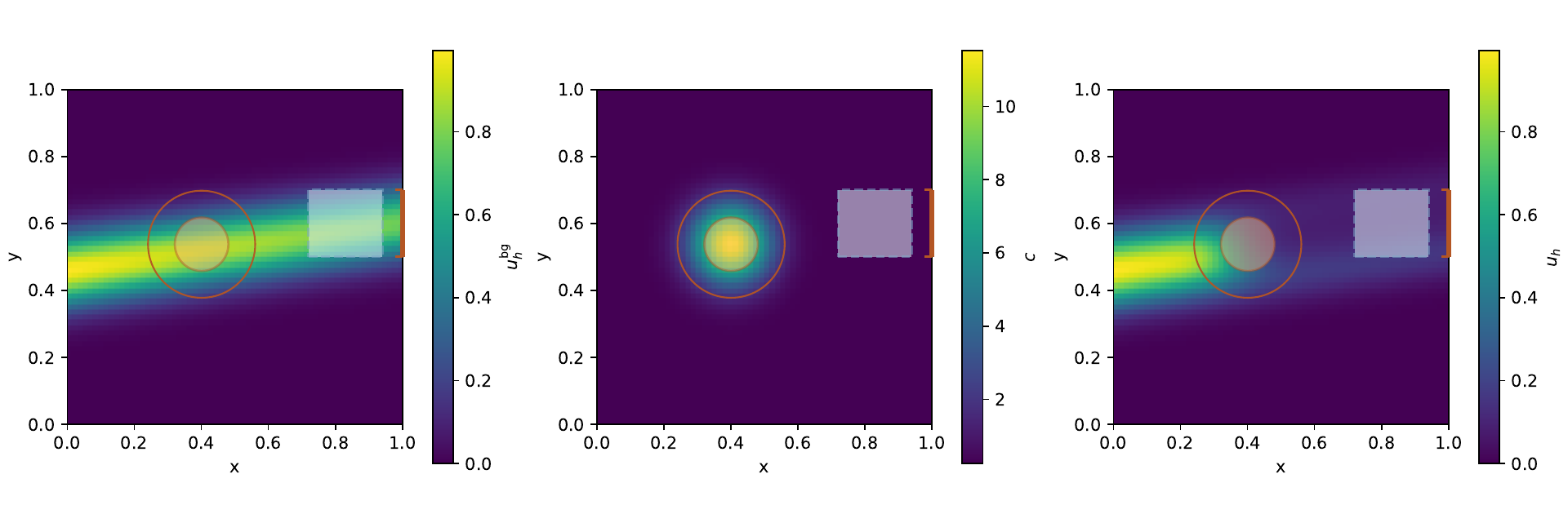}
\caption{\CurrentExample: representative material-shadowing
realisation. Left: the background solution
$u_h^{\mathrm{bg}}$ with $c=c_{\mathrm{bg}}$. Centre: the absorption
coefficient $c(x,\zeta)$. Right: the corresponding transported
solution $u_h(x,\zeta)$. The circle marks the uncertain inclusion, the
rectangle marks the target region $D_T$, and the boundary marker shows
the detector interval.}
\label{fig:shadowing_representative}
\end{figure}

\subsubsection{Observables and ensemble diagnostics.}

Let $u_h^{\mathrm{bg}}$ be the solution with
$c=c_{\mathrm{bg}}$. The relative attenuation and shadow probability
are
\begin{align}
A_h(x,\zeta)
&=
\operatorname{clip}_{\qb{0,1}}
\qp{
\frac{u_h^{\mathrm{bg}}(x)-u_h(x,\zeta)}
{\max\qc{u_h^{\mathrm{bg}}(x),\varepsilon_{\mathrm{bg}}}}
},
\qquad
\varepsilon_{\mathrm{bg}}=10^{-8},
\label{eq:shadow_relative_attenuation}
\\
P_{\mathrm{sh}}(x)
&=
\mathbb P\qp{A_h(x,\zeta)>\tfrac{1}{10}}.
\label{eq:shadow_probability}
\end{align}

The detector and target geometry are
\begin{align}
y_{\mathrm{det}}
&=y_b+\tan\theta,
&
\Delta y_{\mathrm{det}}
&=0.20,
\label{eq:shadow_detector_geometry}
\\
D_T
&=
\qb{0.72,0.94}
\times
\qb{y_{\mathrm{det}}-0.10,y_{\mathrm{det}}+0.10}.
\label{eq:shadow_target_region}
\end{align}
The principal responses are
\begin{align}
J_{\mathrm{det}}(u;\zeta)
&=
\int_{y_{\mathrm{det}}-\Delta y_{\mathrm{det}}/2}
     ^{y_{\mathrm{det}}+\Delta y_{\mathrm{det}}/2}
b_xu(1,y,\zeta)\,\mathrm dy,
\label{eq:shadow_detector}
\\
J_T(u;\zeta)
&=
\frac{1}{\norm{D_T}}
\int_{D_T}u(x,\zeta)\,\mathrm dx,
\label{eq:shadow_target_mean}
\\
J_{\mathrm{inc}}(u;\zeta)
&=
\int_D
\qp{c(x,\zeta)-c_{\mathrm{bg}}}
u(x,\zeta)\,\mathrm dx.
\label{eq:shadow_inclusion_absorption}
\end{align}

To relate the four uncertain inclusion parameters to a scalar measure
of beam overlap, define the inclusion optical depth along the
unperturbed beam centreline by
\begin{equation}
\tau_{\mathrm{beam}}(\zeta)
=
\frac{1}{b_x}
\int_0^1
\qb{
c\qp{x,y_b+x\tan\theta,\zeta}-c_{\mathrm{bg}}
}
\,\mathrm dx.
\label{eq:shadow_beamline_optical_depth}
\end{equation}
The quantity $\tau_{\mathrm{beam}}$ incorporates the position, width
and strength of the inclusion through its overlap with the beam.

Figure~\ref{fig:shadowing_ensemble_story} connects the representative
realisation to the full ensemble. The shadow probability occupies a
broad downstream region, while the largest standard deviation occurs
near the transition between weakly and strongly attenuated states.
The detector current decreases monotonically with
$\tau_{\mathrm{beam}}$; their sample correlation is $-0.949$. Thus,
although four parameters define the inclusion, their effect on this
detector is largely organised by a single physically meaningful
overlap diagnostic.

\begin{figure}[h!]
\centering

\begin{subfigure}[t]{0.315\textwidth}
\centering
\includegraphics[width=\textwidth]
{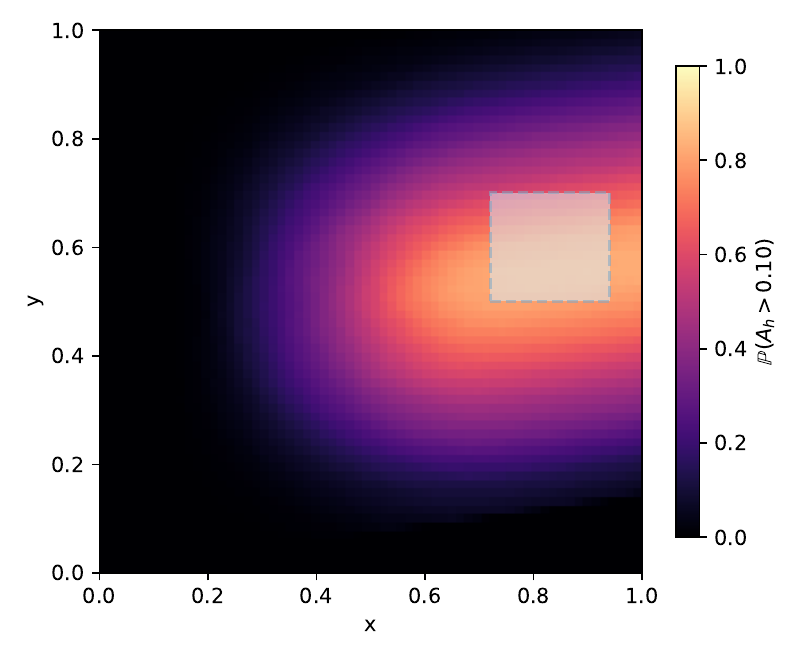}
\caption{$\mathbb P(A_h>0.10)$.}
\label{fig:shadowing_probability}
\end{subfigure}
\hfill
\begin{subfigure}[t]{0.315\textwidth}
\centering
\includegraphics[width=\textwidth]
{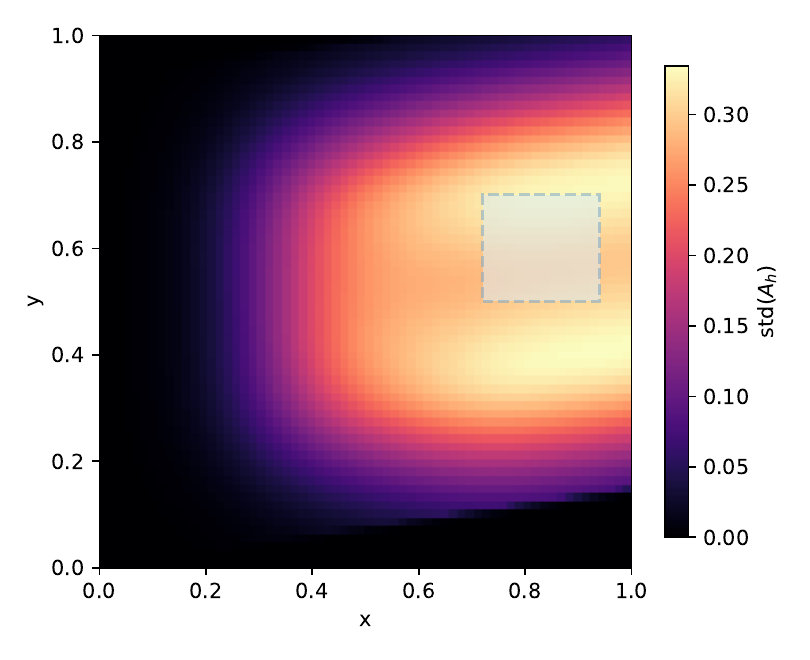}
\caption{$\operatorname{std}(A_h)$.}
\label{fig:shadowing_std}
\end{subfigure}
\hfill
\begin{subfigure}[t]{0.315\textwidth}
\centering
\includegraphics[width=\textwidth]
{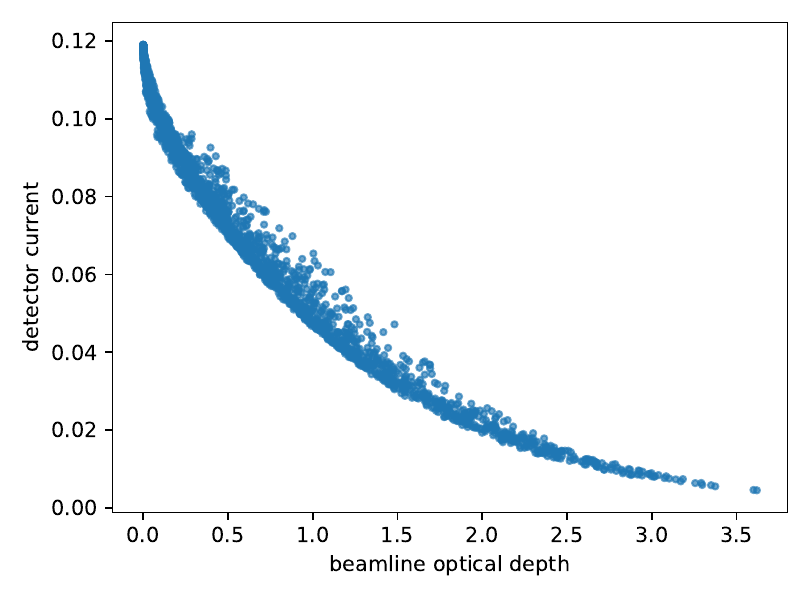}
\caption{$J_{\mathrm{det}}$ against $\tau_{\mathrm{beam}}$.}
\label{fig:shadowing_detector_optical_depth}
\end{subfigure}

\caption{\CurrentExample: ensemble-level material-shadowing
diagnostics. Panel~\textup{(a)} shows the probability of at least
ten percent relative attenuation, panel~\textup{(b)} shows the sample
standard deviation of the relative attenuation, and
panel~\textup{(c)} relates the detector current to the inclusion
optical depth along the incident beam centreline. The rectangle in
panels~\textup{(a)} and~\textup{(b)} marks the target region $D_T$.}
\label{fig:shadowing_ensemble_story}
\end{figure}

\begin{table}[h!]
\centering
\scriptsize
\setlength{\tabcolsep}{5pt}
\begin{tabular}{ccccc}
\hline
Quantity
& mean
& standard deviation
& CV
& 5--95\% range
\\
\hline
$J_{\mathrm{det}}$
& $6.446\times10^{-2}$
& $3.138\times10^{-2}$
& $0.487$
& $\qb{0.014,0.114}$
\\
$J_T$
& $3.343\times10^{-1}$
& $1.594\times10^{-1}$
& $0.477$
& $\qb{0.078,0.583}$
\\
$J_{\mathrm{inc}}$
& $7.162\times10^{-2}$
& $4.030\times10^{-2}$
& $0.563$
& $\qb{9.726\times10^{-3},0.139}$
\\
\hline
\end{tabular}
\caption{\CurrentExample: sample statistics for the principal scalar
responses, computed from $4096$ samples.}
\label{tab:shadowing_statistics}
\end{table}

\subsubsection{Pathwise sensitivities and adjoint verification.}

Derivative-based scores \cite{SobolKucherenko2009DGSM,Cacuci2003} are
computed for $J_{\mathrm{det}}$, $J_T$ and $J_{\mathrm{inc}}$ using
$N_{\mathrm{DGSM}}=2048$ samples and
\eqref{eq:uq_dgsm_estimator}--\eqref{eq:uq_dgsm_score}.  Derivatives
are taken with respect to the standardised variables $\zeta_r$,
including the parameter scalings in
\eqref{eq:shadow_standardised_centres}--%
\eqref{eq:shadow_standardised_width_strength}. These quantities are
normalised mean-square pathwise derivative diagnostics.

\begin{table}[h!]
\centering
\scriptsize
\setlength{\tabcolsep}{7pt}
\begin{tabular}{ccccc}
\hline
Observable
& $\zeta_1\;(x_c)$
& $\zeta_2\;(y_c)$
& $\zeta_3\;(\sigma_c)$
& $\zeta_4\;(a_{\mathrm{inc}})$
\\
\hline
$J_{\mathrm{det}}$ & $0.012$ & $0.732$ & $0.184$ & $0.072$ \\
$J_T$              & $0.016$ & $0.731$ & $0.180$ & $0.073$ \\
$J_{\mathrm{inc}}$ & $0.014$ & $0.680$ & $0.219$ & $0.087$ \\
\hline
\end{tabular}
\caption{\CurrentExample: normalised derivative-based scores for the
principal responses. Derivatives are taken with respect to the
standardised input variables using reverse-mode differentiation
through the wavefront sweep.}
\label{tab:shadowing_dgsm}
\end{table}

The ordering is stable across the three responses. The vertical
position $y_c$ contributes scores between $0.680$ and $0.732$,
followed by the inclusion width $\sigma_c$, whose scores lie between
$0.180$ and $0.219$. The strength $a_{\mathrm{inc}}$ contributes between
$0.072$ and $0.087$, while the horizontal position $x_c$ contributes
only between $0.012$ and $0.016$. This ordering is consistent with
Figure~\ref{fig:shadowing_ensemble_story}: vertical overlap with the
narrow incident beam dominates the response, whereas translation
along the beam direction has comparatively little effect.

As an independent verification of the differentiable implementation,
we compare the wavefront solution and its reverse-mode derivatives
with separately assembled block forward and transpose solves.  The
comparison is performed in double precision on
$N_x=N_y\in\{8,12,16,24,32\}$ for the three principal observables
$J_{\mathrm{det}}$, $J_T$ and $J_{\mathrm{inc}}$. The maximum errors
over all meshes and observables are less than double-precision
roundoff.

\FloatBarrier

\ExampleSubsection{Fixed-work coupled multigroup power-iteration stress test}
{subsec:pincell_stress_test}

The final experiment embeds the batched sweep in a seven-group
transport eigenvalue iteration and tests simultaneous tensorisation
over samples, energy groups and sign-compatible ordinates, together
with sample microbatching inside a coupled source iteration. The
workload uses a C5G7/KAIST-style material and geometry construction to
exercise the coupled algorithm at substantially larger scale.

All layouts start from the same initial state and execute exactly
$1000$ power iterations. The common outer-iteration count gives
identical computational work across layouts and isolates three
features, simultaneous sample--group--ordinate execution, the
memory--throughput tradeoff under sample microbatching, and agreement
of the resulting terminal diagnostics across microbatch sizes.

\subsubsection{Coupled calculation and material configuration.}

The proqblem uses a $102\times102$ pin map, one $\mathbb Q_0$ cell per
pin, $N_{\mathrm{samp}}=1536$ samples, $G=7$ energy groups and $M=16$
ordinates. Hence
\[
N_{\mathrm{cell}}
=
102^2
=
10{,}404,
\qquad
N_{\mathrm{samp}}N_{\mathrm{cell}}GM
=
1{,}789{,}820{,}928
\]
sample--cell--group--ordinate updates are performed per power
iteration. 

The synthetic material data retain the seven-group structure and
material categories of the C5G7 MOX benchmark
\cite{LewisSmithNa2001C5G7,SmithLewisNa2005C5G73D,Dahl2006PARTISN},
including UO$_2$, three MOX enrichments, guide-tube cells,
fission-chamber cells and moderator.  The reflected quarter-core
construction follows KAIST-style geometry conventions
\cite{Cho2000KAIST1A,KnightBryceHall2013}, with an additional
moderator surround and vacuum exterior inflow for the present
algorithmic stress test.

\begin{figure}[h!]
\centering
\includegraphics[width=0.58\textwidth]
{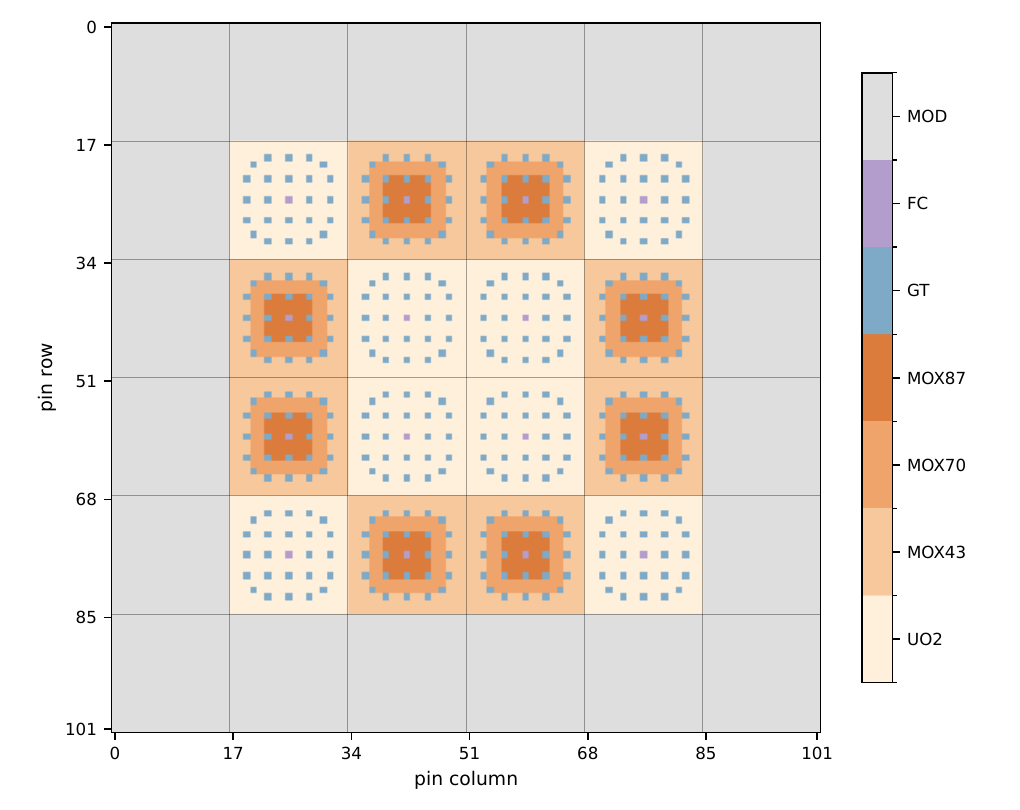}
\caption{\CurrentExample: material configuration for the
C5G7/KAIST-style stress test, containing UO$_2$, three MOX
enrichments, guide-tube, fission-chamber and moderator cells.}
\label{fig:pincell_material_map}
\end{figure}

The quadrature has four equally weighted ordinates in each quadrant
and uses the normalisation
\eqref{eq:model_angular_weight_normalisation}. The four sign classes
use separate Cartesian wavefronts. Within each class, samples, groups
and its four ordinates are batched together. Each class has $203$
stages, and its central stage exposes
\[
1536\times7\times4\times102
=
4{,}386{,}816
\]
independent cell updates.

At iteration $\ell$, the scattering and fission source is
\begin{equation}
Q_{\gamma,m}^{(\ell)}(x,\xi)
=
\sum_{\gamma'=1}^{G}
\Sigma_{s,\gamma'\to\gamma}(x,\xi)
\phi_{\gamma'}^{(\ell)}(x,\xi)
+
\frac{\chi_\gamma(x)}{k^{(\ell)}(\xi)}
\sum_{\gamma'=1}^{G}
\nu\Sigma_{f,\gamma'}(x,\xi)
\phi_{\gamma'}^{(\ell)}(x,\xi).
\label{eq:pincell_power_source}
\end{equation}
Once this source has been formed, the fixed-source equations are
independent over samples, groups and ordinates within each sign class.

For sample $n$, define
\begin{equation}
F_{n,K}^{(\ell)}
=
\sum_{\gamma=1}^{G}
\int_K
\nu\Sigma_{f,\gamma}(x,\xi^{(n)})
\phi_\gamma^{(\ell)}(x,\xi^{(n)})\,\mathrm dx,
\qquad
\mathscr F_n^{(\ell)}
=
\sum_K F_{n,K}^{(\ell)}.
\label{eq:pincell_cellwise_fission}
\end{equation}
The samplewise update is
\begin{align}
k_n^{(\ell+1)}
&=
k_n^{(\ell)}
\widetilde{\mathscr F}_n^{(\ell+1)},
\label{eq:pincell_keff_update}
\\
\phi_{\gamma,n}^{(\ell+1)}
&=
\frac{\widetilde\phi_{\gamma,n}^{(\ell+1)}}
{\widetilde{\mathscr F}_n^{(\ell+1)}},
\label{eq:pincell_flux_normalisation}
\\
F_{n,K}^{(\ell+1)}
&=
\frac{\widetilde F_{n,K}^{(\ell+1)}}
{\widetilde{\mathscr F}_n^{(\ell+1)}}.
\label{eq:pincell_fission_normalisation}
\end{align}
Thus each sample has unit total fission production after every
iteration.

The nominal fission-production scale is calibrated once. Let
$k_{\mathrm{trial}}$ be the multiplication factor obtained with all
perturbations zero and unit scale, then
\begin{equation}
s_f=\frac{1}{k_{\mathrm{trial}}}.
\label{eq:pincell_fission_calibration}
\end{equation}
For the reported calculation, $s_f=15.672$, and this value is fixed
across the ensemble.

\subsubsection{Structured uncertainty model.}

The six inputs are independent standard-normal variables,
\begin{equation}
\xi=(z_1,\ldots,z_6),
\qquad
z_r\stackrel{\mathrm{ind}}{\sim}\mathcal N(0,1).
\label{eq:pincell_random_variables}
\end{equation}
For a nonnegative coefficient $q_{\gamma,K}^0$ on an active set
$\mathcal A$, conservative perturbations have the form
\begin{align}
q_{\gamma,K}(z)
&=
\frac{
q_{\gamma,K}^0
\exp(azs_{\gamma,K})
}{
Z_{\mathcal A}(z)
},
\qquad
(\gamma,K)\in\mathcal A,
\label{eq:pincell_conservative_transform}
\\
Z_{\mathcal A}(z)
&=
\frac{
\displaystyle
\sum_{(\gamma,K)\in\mathcal A}
q_{\gamma,K}^0\exp(azs_{\gamma,K})
}{
\displaystyle
\sum_{(\gamma,K)\in\mathcal A}q_{\gamma,K}^0
}.
\label{eq:pincell_conservative_normalisation}
\end{align}
Consequently,
\begin{equation}
\sum_{(\gamma,K)\in\mathcal A}q_{\gamma,K}(z)
=
\sum_{(\gamma,K)\in\mathcal A}q_{\gamma,K}^0.
\label{eq:pincell_conservative_identity}
\end{equation}

Entries outside $\mathcal A$ are unchanged. The transformations and
amplitudes are summarised in
Table~\ref{tab:pincell_perturbations}. The variable $z_1$ transfers
thermal absorption between UO$_2$ and MOX fuel, while $z_2$
redistributes MOX fission production from MOX43 towards MOX87. The
variable $z_3$ tilts the fuel fission spectrum using
\[
s_{\mathrm{spec}}(\gamma)
=
-1+\frac{2(\gamma-1)}{G-1}.
\]
The variable $z_4$ redistributes moderator scattering from
within-group to downscatter entries while preserving each
scattering-row sum. Finally, $z_5$ redistributes thermal fuel
absorption from the core centre towards its periphery, and $z_6$
redistributes fuel fission production from left to right using
\[
s_{\mathrm{hor}}(K)=2x_{2,K}-1.
\]

The removal and production perturbations follow the convention of
Section~\ref{sec:model}. The calibrated factor $s_f$ and the variables
$z_2$, $z_3$ and $z_6$ act on the fission-production coefficient
$\nu\Sigma_f$ and leave the removal term fixed. The variables $z_1$ and
$z_5$ perturb non-scattering removal, so the total cross section is
updated according to
\begin{equation}
\Sigma_{t,\gamma,K}(\xi)
=
\Sigma_{t,\gamma,K}^{0}
+
\Sigma_{a,\gamma,K}(\xi)
-
\Sigma_{a,\gamma,K}^{0}.
\label{eq:pincell_total_cross_section_update}
\end{equation}
The variable $z_4$ redistributes moderator scattering while preserving
each scattering-row sum and therefore leaves the total scattering
removal unchanged.

The implementation checks the preserved absorption totals, integrated
fission coefficients and moderator scattering-row sums for every
sample. The maximum relative discrepancies are
$6.582\times10^{-16}$, $7.013\times10^{-16}$ and
$4.108\times10^{-16}$, respectively.

\begin{table}[h!]
\centering
\scriptsize
\setlength{\tabcolsep}{4pt}
\begin{tabular}{cccc}
\hline
Variable & transformation & amplitude & preserved coefficient \\
\hline
$z_1$ & UO2/MOX thermal absorption & $0.060$ & thermal fuel $\Sigma_a$ \\
$z_2$ & MOX enrichment redistribution & $0.080$ & MOX $\nu\Sigma_f$ \\
$z_3$ & fuel fission spectral tilt & $0.100$ & material $\nu\Sigma_f$ \\
$z_4$ & moderator downscatter split & $0.120$ & scattering-row sum \\
$z_5$ & radial fuel absorption & $0.075$ & thermal fuel $\Sigma_a$ \\
$z_6$ & horizontal fuel fission tilt & $0.100$ & fuel $\nu\Sigma_f$ \\
\hline
\end{tabular}
\caption{\CurrentExample: independent standard-normal spatial and spectral perturbations.}
\label{tab:pincell_perturbations}
\end{table}

\subsubsection{Fixed-work iteration monitoring and microbatch scaling.}

We monitor the outer iteration using the worst-sample eigenvalue and
fission-shape residuals
\begin{align}
r_k^{(\ell)}
&=
\max_n
\frac{
\norm{k_n^{(\ell)}-k_n^{(\ell-1)}}
}{
\max\qc{\norm{k_n^{(\ell)}},10^{-14}}
},
\label{eq:pincell_keff_residual}
\\
r_F^{(\ell)}
&=
\max_n
\frac{
\Norm{ F_n^{(\ell)}-F_n^{(\ell-1)}}_1
}{
\max\qc{\Norm{ F_n^{(\ell)}}_1,10^{-14}}
}.
\label{eq:pincell_fission_residual}
\end{align}
The full-batch run reaches the iteration limit, with
\[
r_k^{(1000)}
=
3.276\times10^{-7},
\qquad
r_F^{(1000)}
=
4.086\times10^{-5}.
\]
At iteration $1000$, the eigenvalue residual is below its nominal
tolerance, while the fission-shape residual is
$4.086\times10^{-5}$. These values characterise the common fixed-work
endpoint used in the layout comparison.

The run performs
\[
1000\times1{,}789{,}820{,}928
=
1{,}789{,}820{,}928{,}000
\]
updates in $2.016\times10^3$ seconds, giving an overall throughput of
$8.878\times10^8$ updates per second. Of the total runtime,
$1.934\times10^3$ seconds are spent in the sweeps and $81.514$ seconds
in source formation, normalisation and updates.

The active sign-class array contains
\[
1536\times7\times4\times10{,}404
=
447{,}455{,}232
\]
single-precision entries, or approximately $1.790\times10^3$ MB.
Processing one sign class at a time avoids storing all $16$
directional states simultaneously.

The same $1536$ samples are processed with
\begin{equation}
B\in\qc{64,256,512,1536}.
\label{eq:pincell_microbatch_sizes}
\end{equation}
Every microbatch reaches the common limit of $1000$ iterations, so all
four layouts perform the same total work.

For fuel pin $\mathcal P_j$, define the terminal pin production for
sample $n$ and layout $B$ by
\begin{equation}
P_{n,j}^{(B)}
=
\sum_{\gamma=1}^{G}
\int_{\mathcal P_j}
\nu\Sigma_{f,\gamma}(x,\xi^{(n)})
\phi_{\gamma,n}^{(B)}(x)\,\mathrm dx.
\label{eq:pincell_pin_power}
\end{equation}
Let $\mathcal I_{\mathrm{fuel}}$ index the fuel pins. The normalised pin
production and its samplewise maximum are
\begin{equation}
\overline P_{n,j}^{(B)}
=
\frac{
P_{n,j}^{(B)}
}{
\norm{\mathcal I_{\mathrm{fuel}}}^{-1}
\displaystyle
\sum_{\ell\in\mathcal I_{\mathrm{fuel}}}
P_{n,\ell}^{(B)}
},
\qquad
P_{\max,n}^{(B)}
=
\max_{j\in\mathcal I_{\mathrm{fuel}}}
\overline P_{n,j}^{(B)}.
\label{eq:pincell_normalised_pin_power}
\end{equation}
Thus $\overline P_{n,j}^{(B)}=1$ denotes the samplewise mean fuel-pin
production, while $P_{\max,n}^{(B)}$ provides a scalar fingerprint that
is sensitive to the spatial distribution of the terminal fission
production.

\begin{table}[h!]
\centering
\scriptsize
\setlength{\tabcolsep}{4pt}
\begin{tabular}{cccc}
\hline
$B$ & time (s) & updates/s & peak alloc. (MB) \\
\hline
$64$ & $1.052\times 10^{4}$ & $1.702\times 10^{8}$ & $525.032$ \\
$256$ & $3.468\times 10^{3}$ & $5.162\times 10^{8}$ & $2.099\times 10^{3}$  \\
$512$ & $2.574\times 10^{3}$ & $6.954\times 10^{8}$ & $4.198\times 10^{3}$  \\
$1536$ & $2.016\times 10^{3}$ & $8.878\times 10^{8}$ & $1.259\times 10^{4}$ \\
\hline
\end{tabular}
\caption{\CurrentExample: sample-microbatch timing and memory.}
\label{tab:pincell_microbatching}
\end{table}

Increasing $B$ from $64$ to $1536$ reduces the runtime from
$1.052\times10^4$ to $2.016\times10^3$ seconds and increases
throughput from $1.702\times10^8$ to $8.878\times10^8$ updates per
second, a factor of approximately $5.22$. The peak allocation
increases by a factor of approximately $24$, from $525$ MB to
$1.259\times10^4$ MB. Both terminal discrepancies are zero for every
tested layout.

\begin{figure}[h!]
\centering

\begin{subfigure}[t]{0.48\textwidth}
\centering
\includegraphics[width=\textwidth]
{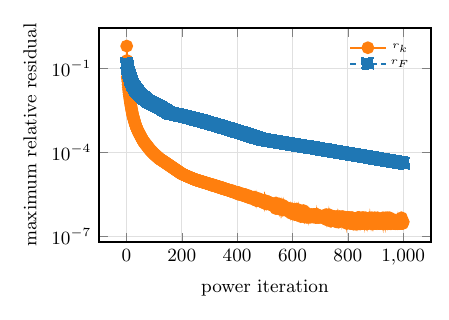}
\caption{Worst-sample iteration residuals.}
\label{fig:pincell_iteration_history}
\end{subfigure}
\hfill
\begin{subfigure}[t]{0.48\textwidth}
\centering
\includegraphics[width=\textwidth]
{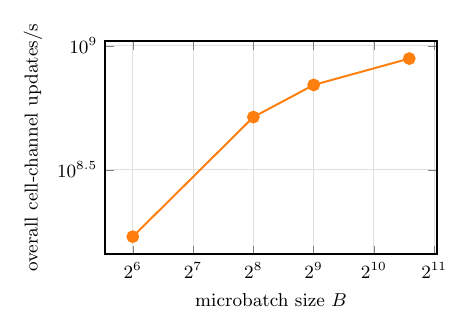}
\caption{Overall throughput against microbatch size.}
\label{fig:pincell_microbatch_throughput}
\end{subfigure}

\caption{\CurrentExample: iteration history and microbatch scaling.
Panel~\textup{(a)} shows the worst-sample eigenvalue and
fission-shape residuals over the common $1000$-iteration work budget.
Panel~\textup{(b)} shows the overall cell--group--ordinate update rate
as the sample microbatch size $B$ is increased. Exact timing, memory
and equivalence results are given in
Table~\ref{tab:pincell_microbatching}.}
\label{fig:pincell_iteration_and_batching}
\end{figure}

Table~\ref{tab:pincell_microbatching} and
Figure~\ref{fig:pincell_iteration_and_batching} give the two
conclusions of the fixed-work stress test. Increasing the sample
microbatch size exposes more channel parallelism and therefore
increases throughput, at the cost of a larger device-memory footprint.

The experiment exercises the complete coupled path from source
formation through sign-class transport sweeps, simultaneous
sample--group--ordinate batching, normalisation and outer iteration.
Across the tested layouts, sample microbatching changes the execution
layout and memory requirement while preserving the reported terminal
fingerprints.

\FloatBarrier

\section{Conclusions}
\label{sec:conclusions}

We have formulated the upwind DG fixed-source transport sweep as a
graph-compatible batched wavefront algorithm. For sweepable problems,
the directed upwind graph gives a block lower-triangular ordering of
the assembled DG system, and the wavefront traversal is algebraically
equivalent to block forward substitution. The tensor formulation
therefore evaluates the same cell-local DG equations and weak inflow
terms while exposing independent work over both wavefront cells and
fixed-source channels.

The principal algorithmic contribution is the sweep-class construction.
Channels whose interior-face signs induce the same dependency graph may
reuse one spatial schedule even when their local matrices, sources and
inflow data differ. Samples, right-hand sides, energy groups and
sign-compatible angular ordinates can therefore be represented as
leading tensor axes and split into memory-controlled microbatches. A
broader schedule-compatible grouping is possible when the union graph
is acyclic, although the present implementation uses graph-equal
classes.

The same finite-dimensional wavefront program provides branchwise
discrete sensitivities. On an open fixed-face-sign branch,
reverse-mode differentiation of the unrolled sweep gives the same
adjoint action as the transpose of the assembled block-triangular DG
system. At face-sign changes the discrete map remains continuous under
the stated assumptions but is only piecewise smooth.

The numerical experiments verify the expected DG convergence rates,
measure the execution-layout and memory--throughput effects of sample
batching, and demonstrate the method in a material-shadowing UQ
calculation. A seven-group fixed-work stress test additionally
demonstrates simultaneous sample--group--ordinate execution and
agreement of global and spatially sensitive terminal diagnostics across
the tested microbatch layouts. The reported performance characterises
the present single-GPU PyTorch implementation.

Future work will extend the formulation to schedule-compatible classes
with nonidentical graphs, factor-based and performance-portable local
solves, unstructured-mesh sweep schedules, and accelerated coupled
iterations.

\section*{Acknowledgements}

TP is supported by the EPSRC programme grant Mathematics of Radiation
Transport (MaThRad) EP/W026899/2 and the Leverhulme Trust
RPG-2021-238. TP is also grateful to Paul Smith who seeded the idea of
this work at the Answers Seminar. All of this support is gratefully
acknowledged.

\section*{Data Availability}

The Python code that reproduces all numerical experiments and plots in
this paper can be found at in a zenodo repository upon acceptance of the publication. 

\section*{Competing Interests}

The author declares no competing interests.

\printbibliography

\end{document}